\documentclass[a4,10pt]{amsart}
\usepackage{amssymb,amstext,amsmath,amscd,amsthm,amsfonts,enumerate,graphicx,latexsym,color}
\usepackage[all]{xy}
\usepackage{cite}
\newtheorem{thm}{Theorem}[section]
\newtheorem*{thm*}{Theorem}

\newtheorem{cor}[thm]{Corollary}
\newtheorem{lem}[thm]{Lemma}
\newtheorem{prop}[thm]{Proposition}
\theoremstyle{definition}
\newtheorem{dfn}[thm]{Definition}
\newtheorem*{dfn*}{Definition}
\newtheorem{rem}[thm]{Remark}
\newtheorem{ques}[thm]{Question}

\newtheorem{ex}[thm]{Example}

\theoremstyle{remark}
\newtheorem*{conv}{Convention}
\newtheorem*{ac}{Acknowledgments}
\newtheorem{claim}{Claim}
\newtheorem*{claim*}{Claim}
\renewcommand{\qedsymbol}{$\blacksquare$}
\numberwithin{equation}{thm}
\def\Hom{\mathsf{Hom}}
\def\lhom{\mathsf{\underline{Hom}}}
\def\End{\mathsf{End}}
\def\lend{\mathsf{\underline{End}}}
\def\Ext{\mathsf{Ext}}
\def\Tor{\mathsf{Tor}}
\def\syz{\mathsf{\Omega}}
\def\tr{\mathsf{Tr}}
\def\P{\mathsf{P}}
\def\ann{\operatorname{\mathsf{Ann}}}
\def\cm{\mathsf{CM}}
\def\lcm{\mathsf{\underline{CM}}}
\def\db{\mathsf{D^b}}
\def\ds{\mathsf{D_{sg}}}
\def\kb{\mathsf{K^b}}
\def\proj{\operatorname{\mathsf{proj}}}
\def\Mod{\operatorname{\mathsf{Mod}}}

\def\mod{\operatorname{\mathsf{mod}}}
\def\lmod{\operatorname{\mathsf{\underline{mod}}}}
\def\gp{\operatorname{\mathsf{Gproj}}}
\def\lgp{\operatorname{\mathsf{\underline{Gproj}}}}
\def\add{\operatorname{\mathsf{add}}}
\def\res{\operatorname{\mathsf{res}}}
\def\K{\mathrm{K}}
\def\A{\mathcal{A}}
\def\X{\mathcal{X}}
\def\Y{\mathcal{Y}}
\def\M{\mathcal{M}}

\def\C{\mathcal{C}}
\def\SS{\mathcal{S}}
\def\T{\mathcal{T}}

\def\m{\mathfrak{m}}
\def\p{\mathfrak{p}}

\def\Z{\mathbb{Z}}
\def\a{\mathsf{A}}
\def\d{\mathsf{D}}
\def\e{\mathsf{E}}
\def\t{\mathsf{T}}
\def\g{\mathbf{G}}

\def\ab{\mathbf{Ab}}

\def\codim{\operatorname{\mathsf{codim}}}
\def\Gpd{\operatorname{\mathsf{Gpd}}}
\def\Im{\operatorname{\mathsf{Im}}}
\def\Ker{\operatorname{\mathsf{Ker}}}

\def\dim{\operatorname{\mathsf{dim}}}
\def\depth{\operatorname{\mathsf{depth}}}
\def\cidim{\operatorname{\mathsf{CIdim}}}
\def\gdim{\operatorname{\mathsf{Gdim}}}
\def\pd{\operatorname{\mathsf{pd}}}
\def\grade{\operatorname{\mathsf{grade}}}
\def\xx{\text{\boldmath{$x$}}}

\def\cx{\operatorname{\mathsf{cx}}}
\def\ca{\mathbb{(A)}}
\def\cb{\mathbb{(B)}}
\def\cc{\mathbb{(C)}}
\def\cd{\mathbb{(D)}}
\def\ce{\mathbb{(E)}}
\def\zero{\mathbf{0}}
\def\ass{\operatorname{\mathsf{Ass}}}
\def\height{\operatorname{\mathsf{ht}}}
\def\spec{\operatorname{\mathsf{Spec}}}
\def\QQ{\mathbb{Q}}
\def\sus{\mathsf{\Sigma}}
\begin{document}
\title[Singularity categories for resolving subcategories]{Singularity categories and singular equivalences\\
for resolving subcategories}
\author{Hiroki Matsui}
\address{Graduate School of Mathematics, Nagoya University, Furocho, Chikusaku, Nagoya, Aichi 464-8602, Japan}
\email{m14037f@math.nagoya-u.ac.jp}
\author{Ryo Takahashi}
\address{Graduate School of Mathematics, Nagoya University, Furocho, Chikusaku, Nagoya, Aichi 464-8602, Japan}
\email{takahashi@math.nagoya-u.ac.jp}
\urladdr{http://www.math.nagoya-u.ac.jp/~takahashi/}
\thanks{RT was partly supported by JSPS Grant-in-Aid for Scientific Research (C) 25400038}
\dedicatory{Dedicated to Professor Yuji Yoshino on the occasion of his sixtieth birthday}
\begin{abstract}
Let $\X$ be a resolving subcategory of an abelian category.
In this paper we investigate the singularity category $\ds(\underline\X)=\db(\mod\underline\X)/\kb(\proj(\mod\underline\X))$ of the stable category $\underline\X$ of $\X$.
We consider when the singularity category is triangle equivalent to the stable category of Gorenstein projective objects, and when the stable categories of two resolving subcategories have triangle equivalent singularity categories.
Applying this to the module category of a Gorenstein ring, we prove that the complete intersections over which the stable categories of resolving subcategories have trivial singularity categories are the simple hypersurface singularities of type $(\a_1)$.
We also generalize several results of Yoshino on totally reflexive modules.
\end{abstract}
\maketitle
\section{Introduction}\label{sect1}

One of the most classical subjects in representation theory of algebras is the study of a stable equivalence of selfinjective algebras, i.e., a triangle equivalence between the stable categories of finitely generated modules over those algebras.
This is extended to (infinite-dimensional) Iwanaga-Gorenstein rings \cite{EJ} under restriction to Cohen-Macaulay modules; a stable equivalence of Iwanaga-Gorenstein rings is a triangle equivalence between the stable categories of Cohen-Macaulay modules over those rings.

Let $R$ be a noetherian ring.
The {\em singularity category} of $R$ is by definition the Verdier quotient
$$
\ds(R)=\db(\mod R)/\kb(\proj(\mod R)),
$$
where $\mod R$ denotes the category of finitely generated $R$-modules, $\db(-)$ the bounded derived category and $\kb(-)$ the bounded homotopy category.
The singularity category $\ds(R)$ is a triangulated category, which is also called the {\em stable (stabilized) derived category}, {\em triangulated category of singularities} and {\em singular derived category}.
This has been introduced by Buchweitz \cite{B} in the 1980s, and in recent years it has been related to the mirror symmetry by Orlov \cite{O}.
A lot of studies on singularity categories have been done in various approaches; see \cite{Bec,BK,Che,FU,IW,Kr,Mu,PV,S,kos,U,ZZ} for instance.
A celebrated theorem of Buchweitz \cite{B} shows that if $R$ is an Iwanaga-Gorenstein ring, then the stable category of Cohen-Macaulay $R$-modules is triangle equivalent to the singularity category of $R$.
Thus, a stable equivalence of Iwanaga-Gorenstein rings is extended to arbitrary noetherian rings as a triangle equivalence between their singularity categories, which is called a {\em singular equivalence}.

Let $\C$ be an additive category with pseudokernels \cite{Au3}.
We denote by $\mod\C$ the category of finitely presented right $\C$-modules, which is a full subcategory of the functor category of $\C$ (see Definition \ref{2.3}) and turns out to be an abelian category with enough projective objects.
This category has been introduced by Auslander \cite{Au3}, and in a series of papers \cite{AR8,AR4,AR5,AR6,AR3,AR7} Auslander and Reiten have established a lot of deep studies on this category.
The notion of singular equivalences of noetherian rings is further extended to additive categories $\C$ by using $\mod\C$ as follows:
We take the Verdier quotient
$$
\ds(\C)=\db(\mod\C)/\kb(\proj(\mod\C)),
$$
and call this the {\em singularity category} of $\C$.
We say that two additive categories $\C,\C'$ are {\em singularly equivalent} if there exists a triangle equivalence $\ds(\C)\cong\ds(\C')$.
Thus, one of the most natural and fundamental questions is the following.
\begin{ques}
When are given two additive categories singularly equivalent?
\end{ques}
The main purpose of this paper is to explore this question for certain additive categories in relation to the {\em Gorenstein projective} property in an abelian category (see Definition \ref{33}).
In fact, we shall find a lot of singular equivalences, which should be interesting and remarkable once one takes into account the fact that triangle equivalences seldom arise over commutative rings.
To state our results, let us recall and introduce several notions.

We say that $\C$ is {\em regular} (respectively, {\em Gorenstein}) if every object of $\mod\C$ has finite projective dimension (respectively, Gorenstein projective dimension).
A regular (respectively, Gorenstein) category $\C$ is called {\em of dimension at most $n$} if every object of $\mod\C$ has projective dimension (respectively, Gorenstein projective dimension) at most $n$, or equivalently, the $n$th syzygies in $\mod\C$ are projective (respectively, Gorenstein projective).
Any abelian (respectively, triangulated) category is regular (respectively, Gorenstein) of dimension at most $2$ (respectively, $0$); see Proposition \ref{referee}.
Note that $\C$ is regular if and only if $\C$ is singularly equivalent to the zero category $\zero$, that is, $\ds(\C)\cong\zero$.
Beligiannis \cite{Be} shows that if $\C$ is Gorenstein of dimension at most $n$ for some $n\ge0$, then $\ds(\C)$ is triangle equivalent to $\lgp(\mod\C)$, which is a generalization of the theorem of Buchweitz stated above.
Here, for an abelian category $\A$ with enough projective objects, $\gp\A$ stands for the full subcategory of Gorenstein projective objects of $\A$.
This is a Frobenius category, and its stable category $\lgp\A$ is a triangulated category \cite{Be}.

A full subcategory of an abelian category with enough projective objects is called {\em resolving} if it contains the projective objects and is closed under direct summands, extensions and kernels of epimorphisms.
This notion has been introduced by Auslander and Bridger \cite{AB} to prove that the category of totally reflexive modules is a resolving subcategory of the category $\mod R$ of finitely generated modules over a noetherian ring $R$.
The category $\cm(R)$ of maximal Cohen-Macaulay modules over a Cohen-Macaulay local ring $R$ is also a resolving subcategory of $\mod R$, and there are many other important subcategories known to be resolving.
The studies of resolving subcategories have been done widely so far; see \cite{APST,AR2,AR,radius,crspd,KS,Sa,res,stcm,arg,crs,Y2} for example.
For a resolving subcategory $\X$ of an abelian category $\A$, let $\underline\X=\X/\proj\A$ be the stable category of $\X$.
Then $\underline\X$ has pseudokernels, and hence $\mod\underline\X$ is an abelian category with enough projective objects \cite{AR}.
In fact, $\underline\X$ still has pseudokernels even if one removes the assumption that $\X$ is closed under direct summands and extensions.
So we define a {\em quasi-resolving} subcategory to be a full subcategory containing the projective objects and closed under finite direct sums and kernels of epimorphisms.
A resolving subcategory is none other than a quasi-resolving subcategory closed under direct summands and extensions.

Now, we recall the following two facts.
The first one is implicitly given by Auslander and Reiten \cite{AR8}, while the second one is essentially obtained by Yoshino \cite{Y2}, based on \cite{AR} (see also \cite{Be2,IO,K}).

\begin{thm}[Auslander-Reiten]\label{AR}
Let $\A$ be an abelian category with enough projective objects.
Let $\X$ be a quasi-resolving subcategory of $\A$.
Suppose that every object in $\X$ has projective dimension at most $n$ in $\A$.
Then $\underline\X$ is regular of dimension at most $3n-1$.
\end{thm}

\begin{thm}[Yoshino]\label{Y}
Let $\A$ be an abelian category with enough projective objects.
Let $\X$ be a quasi-resolving subcategory of $\A$.
Suppose that $\X$ is contained in $\gp\A$ and closed under cosyzygies.
Then $\underline\X$ is Gorenstein of dimension at most $0$, that is, $\mod\underline\X$ is a Frobenius category.
\end{thm}

Motivated by these two theorems, in this paper we first study stable categories of quasi-resolving subcategories.
To be more precise, we shall establish a Gorenstein analogue of Theorem \ref{AR} which extends Theorem \ref{Y}.
Let $\X$ be a quasi-resolving subcategory of an abelian category $\A$ and $n\ge0$ an integer.
Denote by $\syz^n\X$ the full subcategory of $\A$ consisting of $n$th syzygies of objects in $\X$.
We introduce the following condition.
$$
\text{$(\g_n)$\ \ $\syz^n\X$ is contained in $\gp\A$ and closed under cosyzygies.}
$$
A typical example of a quasi-resolving subcategory satisfying $(\g_n)$ is the full subcategory of objects of Gorenstein projective dimension at most $n$.
Every resolving subcategory over a local complete intersection $R$ satisfies $(\g_n)$ for $\A=\mod R$ and $n=\dim R$.
We shall obtain the following theorem.

\begin{thm}\label{thma}
Let $\A$ be an abelian category with enough projective objects.
Let $\X$ be a quasi-resolving subcategory of $\A$ satisfying the condition $(\g_n)$.
Then $\syz^n\X=\X\cap\gp\A$ holds.
Denote this by $\Y$.
\begin{enumerate}[\rm(1)]
\item
$\underline\X$ is Gorenstein of dimension at most $3n$, and there is a triangle equivalence $\ds(\underline\X)\cong\lgp(\mod\underline\X)$.
\item
$\Y$ is a Frobenius subcategory of $\gp\A$, and $\underline\Y$ is a triangulated subcategory of $\lgp\A$.
\item
$\underline\Y$ is Gorenstein of dimension at most $0$, and there is a triangle equivalence $\ds(\underline\Y)\cong\underline{\mod\underline\Y}$.
\item
$\underline\X$ and $\underline\Y$ are singularly equivalent.
\end{enumerate}
\end{thm}

The first assertion of Theorem \ref{thma}(1) is a Gorenstein analogue of Theorem \ref{AR}, and letting $n=0$ recovers Theorem \ref{Y}.
In relation to this result, we also consider when the condition $\Ext_R^{>n}(\X,R)=0$ is equivalent to the condition $\Ext_{\mod\underline\X}^{>3n}(\mod\underline\X,\proj(\mod\underline\X))=0$ for a resolving subcategory $\X$ of modules over a commutative ring $R$.
For $n=0$ the latter condition is nothing but the condition that $\mod\underline\X$ is quasi-Frobenius (in the sense of \cite{Y2}), and we can recover \cite[Theorem 4.2]{Y2}.
One also obtains a characterization of Gorenstein local rings $R$ in terms of modules over the stable category $\lmod R$, whose artinian case is none other than \cite[Corollary 4.3]{Y2}.

The second assertion of Theorem \ref{thma}(1) clarifies the structure of the singularity category $\ds(\underline\X)$; for example, each object of $\ds(\underline\X)$ turns out to be isomorphic to a shift of an object in $\mod\underline\X$.
This is a key to deduce the remaining assertions in Theorem \ref{thma}.

Using Theorem \ref{thma}(4) one can obtain various singular equivalences.
For example, let $\X$ be a quasi-resolving subcategory of $\A$ with $\syz^n\X\subseteq\gp\A\subseteq\X$ for some $n\ge0$, e.g., the category of objects of Gorenstein projective dimension at most $n$.
Then $\underline\X$ is singularly equivalent to $\lgp\A$.
In particular, if $R$ is a commutative Gorenstein local ring and $\X$ is a quasi-resolving subcategory of $\mod R$ containing $\cm(R)$, then $\underline\X$ and $\lcm(R)$ are singularly equivalent.
If $R$ is moreover a complete intersection, then $\underline\X$ is singularly equivalent to $\underline{\X\cap\cm(R)}$ for all resolving subcategories $\X$ of $\mod R$.
Combining this with a classification of resolving subcategories given in \cite{stcm} yields that if $R$ is an isolated hypersurface singularity, then $\underline\X$ is singularly equivalent to either $\lcm(R)$ or the zero category $\zero$, so there are only at most two singular equivalence classes.

We are thus interested in asking when $R$ admits exactly one singular equivalence class.
More precisely, for a resolving subcategory $\X$ of $\mod R$ we consider when $\underline\X$ is singularly equivalent to $\zero$, or equivalently, when $\underline\X$ is regular.
We shall prove the following theorem, which characterizes the regularity of stable categories of resolving subcategories.

\begin{thm}\label{thmc}
Let $R$ be a $d$-dimensional nonregular complete local ring with algebraically closed residue field $k$ of characteristic zero.
Then the following are equivalent.
\begin{enumerate}[\rm(1)]
\item
$R$ is Gorenstein, and $\lcm(R)$ is regular.
\item
$R$ is a complete intersection, and $\underline\X$ is regular for every resolving subcategory $\X$ of $\mod R$.
\item
$R$ is a complete intersection, and $\underline\X$ is regular for some resolving subcategory $\X$ of $\mod R$ that contains a module of maximal complexity.
\item
$R$ is a simple hypersurface singularity of type $(\a_1)$, namely, $R\cong k[[x_0,\dots,x_d]]/(x_0^2+\cdots+x_d^2)$.
\end{enumerate}
When one of these conditions is satisfied, $\lcm(R)$ is regular of dimension at most $0$, namely, $\mod\lcm(R)$ has global dimension $0$.
\end{thm}

If an additive category $\C$ is regular of dimension at most $n$ for some $n\ge0$, then $\C$ is regular by definition.
In view of Theorem \ref{thmc}, as of independent interest, we also consider when the converse of this statement holds for $\C=\underline\X$.

The paper is organized as follows.
In Section \ref{sect2}, for later use we state several fundamental results on the structure of modules over the stable category of a subcategory of an abelian category.
In Sections \ref{sect3} and \ref{sect4}, we compare many kinds of conditions on quasi-resolving subcategories.
The first assertion of Theorem \ref{thma} is shown in this section.
In Section \ref{sect5}, we study singular equivalence of the stable categories of resolving subcategories, and give proofs of the remaining assertions of Theorem \ref{thma}.
In Section \ref{sect6}, we consider when the stable category of a resolving subcategory is regular, and prove Theorem \ref{thmc}.

\begin{conv}
Throughout this paper, unless otherwise specified, we use the following convention: All subcategories are assumed to be strictly full (i.e., full and closed under isomorphism).
Let $\A$ be an abelian category with enough projective objects, and denote by $\proj\A$ the full subcategory of projective objects of $\A$.
Let $\ab$ and $\zero$ stand for the category of abelian groups and the zero category, respectively.
Let $R$ be a commutative noetherian ring, and denote by $\mod R$ the category of finitely generated $R$-modules.
All $R$-modules in this paper are assumed to be finitely generated.
A maximal Cohen-Macaulay module is simply called an MCM module.
When $R$ is a Cohen-Macaulay local ring, $\cm(R)$ stands for the full subcategory of $\mod R$ consisting of MCM modules over $R$.
\end{conv}

\section{Modules over stable categories}\label{sect2}

In this section we give several fundamental results on the structure of modules over the stable category of a subcategory of $\A$.
Most of the results are given in \cite{AR,AR3,AR4,AR5,AR6,AR7,AR8} and \cite{Y2} at least essentially, but we give proofs for the convenience of the reader.
Let us begin with recalling the definition of a stable category.

\begin{dfn}
\begin{enumerate}[(1)]
\item
For two objects $M,N$ of $\A$ we define the {\em stable hom-set} as the quotient group $\lhom_\A(M,N):=\Hom_\A(M,N)/\P_\A(M,N)$, where $\P_\A(M,N)$ consists of all morphisms from $M$ to $N$ that factor through objects in $\proj\A$.
For $\A=\mod R$, we simply write $\lhom_R(M,N)$ for $\lhom_\A(M,N)$.
\item
Let $\X$ be a subcategory of $\A$ containing $\proj\A$.
Then the quotient category $\underline\X:=\X/\proj\A$ is called the {\em stable category} of $\X$; the objects of $\underline\X$ are the same as those of $\X$, and the hom-set $\Hom_{\underline\X}(M,N)$ of $M,N\in\underline\X$ is defined as $\lhom_\A(M,N)$.
Hence $\underline\X$ is a full subcategory of $\underline\A$.
\end{enumerate}
\end{dfn}

The following lemma is given in the case where $\A=\mod R$ by \cite[Lemma 2.7]{Y2}, whose proof uses Auslander transposes of modules and does not work for general abelian categories $\A$.

\begin{lem}\label{half}
Let $0\to A\xrightarrow{f} B\xrightarrow{g} C\to0$ be an exact sequence in $\A$.
Let $X$ be an object of $\A$.
\begin{enumerate}[\rm(1)]
\item
The induced sequence $\lhom_\A(X,A)\to\lhom_\A(X,B)\to\lhom_\A(X,C)$ is exact.
\item
The induced sequence $\lhom_\A(C,X)\to\lhom_\A(B,X)\to\lhom_\A(A,X)$ is exact if $\Ext_\A^1(C,\proj\A)=0$.
\end{enumerate}
\end{lem}

\begin{proof}
(1) The left-exactness of the functor $\Hom$ shows that the induced sequence
$$
0\to \Hom_\A(X,A)\xrightarrow{\Hom_\A(X,f)} \Hom_\A(X,B)\xrightarrow{\Hom_\A(X,g)}  \Hom_\A(X,C)
$$
is exact.
Let $u:X\to B$ be a morphism in $\A$ such that $\underline{u}$ is in the kernel of $\lhom_\A(X,g):\lhom_\A(X,B)\to\lhom_\A(X,C)$.
Then $gu$ is the composition of some morphisms $a:X \to P$ and $b:P \to C$ in $\A$, where $P$ is a projective object in $\A$.
There is a morphism $c:P \to B$ with $gc=b$.
Hence $g(u-ca)=gu-gca=gu-ba=0$, and we find a morphism $d:X\to A$ such that $u-ca=fd$.
We have $\underline{u}=\underline{fd}$, which is in the image of $\lhom_\A(X,f):\lhom_\A(X,A)\to\lhom_\A(X,B)$.

(2) By the left-exactness of $\Hom$ the induced sequence
$$
0\to\Hom_\A(C,X)\xrightarrow{\Hom_\A(g,X)} \Hom_\A(B,X)\xrightarrow{\Hom_\A(f,X)}  \Hom_\A(A,X)
$$
is exact.
Let $u:B\to X$ be a morphism with $\underline{u}\in\Ker\lhom_\A(f,X)$.
Then $uf$ is the composition of some morphisms $a:A \to P$ and $b:P \to X$ with $P\in\proj\A$.
Since $\Ext_\A^1(C,\proj\A)=0$, the map $\Hom_\A(f,P)$ is surjective, and there exists a morphism $c:B \to P$ such that $cf=a$.
We have $(u-bc)f=0$, and find a morphism $d:C\to X$ with $u-bc=dg$.
It follows that $\underline{u}=\underline{dg}\in\Im\lhom(g,X)$.
\end{proof}

Next, we recall the definition of the category of finitely presented modules over an additive category.

\begin{dfn}\label{2.3}
Let $\C$ be an additive category.
Denote by $\Mod\C$ the {\em functor category} of $\C$, that is, the objects are additive contravariant functors from $\C$ to $\ab$, and the morphisms are natural transformations.
An object and a morphism of $\Mod\C$ are called a {\em (right) $\C$-module} and a {\em $\C$-homomorphism}, respectively.
A $\C$-module $F$ is said to be {\em finitely presented} if there is an exact sequence
$$
\Hom_\C(-,X) \to \Hom_\C(-,Y) \to F \to 0
$$
in the abelian category $\Mod\C$ with $X,Y\in\C$.
The full subcategory of $\Mod\C$ consisting of finitely presented $\C$-modules is denoted by $\mod\C$.
This is sometimes called the {\em Auslander category} of $\C$.
\end{dfn}

\begin{rem}\label{yoneda}
\begin{enumerate}[(1)]
\item
For each $X\in\C$ the functor $\Hom_\C(-,X)$ is a projective object of $\mod\C$, and conversely, any projective object of $\mod\C$ is isomorphic to a direct summand of $\Hom_\C(-,X)$ for some $X\in\C$.
Yoneda's lemma asserts that the assignment $X\mapsto\Hom_\C(-,X)$ makes a fully faithful functor
$$
\C\hookrightarrow\mod\C.
$$
This is called the {\em Yoneda embedding} of $\C$.
Thanks to this, any $\C$-homomorphism $\Hom_\C(-,X) \to \Hom_\C(-,Y)$ can be described as $\Hom_\C(-,f)$ for some morphism $f\in\Hom_\C(X,Y)$.
\item
It is said that {\em $\C$ has pseudokernels} if for each morphism $f:X\to Y$ in $\C$ there exists a morphism $g:Z\to X$ in $\C$ such that the induced sequence
$$
\Hom_\C(-,Z)\xrightarrow{\Hom_\C(-,g)}\Hom_\C(-,X)\xrightarrow{\Hom_\C(-,f)}\Hom_\C(-,Y)
$$
is exact.
This condition is equivalent to saying that $\mod\C$ is abelian; see \cite[Chapter III, \S2]{Au2}.
\end{enumerate}
\end{rem}

Here we recall the definition of the syzygies of an object in an abelian category.

\begin{dfn}[cf. Definition \ref{31}]\label{25}
Let $M$ be an object of $\A$, and let $n$ be a positive integer.
Let $f:P\to M$ be a right ($\proj\A$)-approximation, which is by definition a morphism such that every morphism from an object in $\proj\A$ to $M$ factors through $f$.
Then the kernel of $f$ is called the {\em first syzygy} of $M$ and denoted by $\syz M$.
The {\em $n$th syzygy} $\syz^nM$ of $M$ is defined inductively as $\syz(\syz^{n-1}M)$.
We put $\syz^0M:=M$.
For a subcategory $\X$ of $\A$ we denote by $\syz^n\X$ the subcategory of $\A$ consisting of all $n$th syzygies.
Note that $\syz^n\X$ contains $\proj\A$ for $n>0$.
\end{dfn}

\begin{rem}
Let $M$ be an object of $\A$, and let $n$ be a positive integer.
\begin{enumerate}[(1)]
\item
Since $\A$ is assumed to have enough projective objects, the $n$th syzygy of $M$ exists, and one has an exact sequence
$$
0 \to \syz^nM \to P_{n-1} \to \cdots \to P_0 \to M \to 0
$$
in $\A$ with $P_i\in\proj\A$ for $0\le i\le n-1$.
\item
By Schanuel's lemma, the $n$th syzygy of $M$ is uniquely determined up to projective summands, whence in the stable category $\underline\A$ it is uniquely determined up to isomorphism.
\end{enumerate}
\end{rem}

Now we show the proposition below, which induces from each short exact sequence in $\A$ a long exact sequence in a functor category.

\begin{prop}\label{longex}
Let $\sigma:0\to A\xrightarrow{f} B\xrightarrow{g} C\to0$ be a short exact sequence in $\A$.
\begin{enumerate}[\rm(1)]
\item
There is a sequence
$$
[\cdots\xrightarrow{g_2} \syz^2C \xrightarrow{d_1} \syz A \xrightarrow{f_1} \syz B \xrightarrow{g_1} \syz C \xrightarrow{d} A \xrightarrow{f} B \xrightarrow{g} C]
$$
of morphisms in $\A$ such that one can form a short exact sequence from any consective two morphisms (consisting three objects) by adding some projective object to the middle object.
\item
Let $\X$ be a subcategory of $\A$ containing $\proj\A$.
One has an induced long exact sequence in $\Mod\underline\X$:
\begin{align*}
\cdots&\to\lhom_\A(-,\syz^2B)|_{\underline\X}\to\lhom_\A(-,\syz^2C)|_{\underline\X}\to\lhom_\A(-,\syz A)|_{\underline\X}\to\lhom_\A(-,\syz B)|_{\underline\X}\\
&\to\lhom_\A(-,\syz C)|_{\underline\X}\to\lhom_\A(-,A)|_{\underline\X}\to\lhom_\A(-,B)|_{\underline\X}\to\lhom_\A(-,C)|_{\underline\X}\\
&\to\Ext_\A^1(-,A)|_{\underline\X}\to\Ext_\A^1(-,B)|_{\underline\X}\to\Ext_\A^1(-,C)|_{\underline\X}\to\Ext_\A^2(-,A)|_{\underline\X}\to\Ext_\A^2(-,B)|_{\underline\X}\to\cdots.
\end{align*}
\end{enumerate}
\end{prop}

\begin{proof}
(1) Taking a surjection $\pi:P\to C$ with $P\in\proj\A$ and making the pullback diagram of $g$ and $\pi$, we get an exact sequence
$$
0 \to \syz C \xrightarrow{\binom{d}{*}} A\oplus P \xrightarrow{(f,*)} B \to 0.
$$
Iterating this procedure gives rise to exact sequences
\begin{align*}
& 0 \to \syz B \xrightarrow{\binom{g_1}{*}} \syz C\oplus P' \xrightarrow{(d,*)} A \to 0,\qquad
0 \to \syz A \xrightarrow{\binom{f_1}{*}} \syz B\oplus P'' \xrightarrow{(g_1,*)} \syz C \to 0,\\
& 0 \to \syz^2C \xrightarrow{\binom{d_1}{*}} \syz A\oplus P''' \xrightarrow{(f_1,*)} \syz B\ \to 0,\qquad
0 \to \syz^2B \xrightarrow{\binom{g_2}{*}} \syz^2C\oplus P'''' \xrightarrow{(d_1,*)} \syz A\ \to 0,\quad\dots,
\end{align*}
where $P',P'',P''',P'''',\dots$ are projective.
Thus we obtain a sequence
$$
[\cdots\xrightarrow{g_2} \syz^2C \xrightarrow{d_1} \syz A \xrightarrow{f_1} \syz B \xrightarrow{g_1} \syz C \xrightarrow{d} A \xrightarrow{f} B \xrightarrow{g} C]
$$
of morphisms, which is what we want.

(2) Let $X\in\underline\X$.
Using Lemma \ref{half} for the sequence obtained in (1) yields an exact sequence
\begin{align}\label{lhomlhom}
\cdots&\to\lhom_\A(X,\syz^2C)\to\lhom_\A(X,\syz A)\to\lhom_\A(X,\syz B)\\
&\to\lhom_\A(X,\syz C)\to\lhom_\A(X,A)\to\lhom_\A(X,B)\to\lhom_\A(X,C).\notag
\end{align}
On the other hand, there is an exact sequence 
$$
0 \to \Hom_\A(X,A) \to \Hom_\A(X,B) \to \Hom_\A(X,C) \xrightarrow{\delta} \Ext_\A^1(X,A) \to \Ext_\A^1(X,B) \to \cdots.
$$
Let $P$ be a projective object of $\A$, and let $\alpha:X\to P$ and $\beta:P\to C$ be morphisms in $\A$.
Then we have $\delta(\beta\alpha)=\Ext_\A^1(\beta\alpha,A)(\sigma)=\Ext_\A^1(\alpha,A)(\Ext_\A^1(\beta,A)(\sigma))=0$ as $\Ext_\A^1(\beta,A)(\sigma)=0$.
Hence $\delta(\P_\A(X,C))=0$.
Therefore, $\delta$ induces a homomorphism $\delta':\lhom_\A(X,C)\to\Ext_\A^1(X,A)$ such that the following diagram commutes, where the vertical maps are canonical surjections.
$$
\xymatrix{
0\ar[r] & \Hom_\A(X,A)\ar[r]\ar@{->>}[d] & \Hom_\A(X,B)\ar[r]\ar@{->>}[d] & \Hom_\A(X,C)\ar[r]^\delta\ar@{->>}[d] & \Ext_\A^1(X,A)\ar[r] & \Ext_\A^1(X,B)\to\cdots\\
& \lhom_\A(X,A)\ar[r] & \lhom_\A(X,B)\ar[r] & \lhom_\A(X,C)\ar[ru]_{\delta'}
}
$$
Using Lemma \ref{half} and diagram chasing, we obtain an exact sequence
\begin{equation}\label{lhomExt}
\lhom_\A(X,A) \to \lhom_\A(X,B) \to \lhom_\A(X,C) \xrightarrow{\delta'} \Ext_\A^1(X,A) \to \Ext_\A^1(X,B) \to \cdots.
\end{equation}
Splicing \eqref{lhomlhom} and \eqref{lhomExt} yields an exact sequence as in the assertion.
\end{proof}

Let us give the definitions of quasi-resolving and resolving subcategories, which are the main targets studied in this paper.

\begin{dfn}
\begin{enumerate}[(1)]
\item
A subcategory $\X$ of $\A$ is called {\em quasi-resolving} if $\X$ satisfies the following conditions.
\begin{enumerate}[(a)]
\item
$\X$ contains $\proj\A$.
\item
$\X$ is closed under finite direct sums, that is, for a finite number of objects $X_1,\dots,X_n$ in $\X$ the direct sum $X_1\oplus\cdots\oplus X_n$ is in $\X$.
\item
$\X$ is closed under kernels of epimorphisms, that is, for each short exact sequence $0\to L\to M\to N\to0$ in $\A$, if $M$ and $N$ are in $\X$, then so is $L$.
\end{enumerate}
\item
A quasi-resolving subcategory $\X$ of $\A$ is called {\em resolving} if $\X$ satisfies the following conditions.
\begin{enumerate}[(a)]
\item
$\X$ is closed under direct summands, namely, if $X$ is an object in $\X$ and $Y$ is a direct summand of $X$ in $\A$, then $Y$ is in $\X$.
\item
$\X$ is closed under extensions, namely, for each short exact sequence $0\to L\to M\to N\to0$ in $\A$, if $L$ and $N$ are in $\X$, then so is $M$.
\end{enumerate}
For a subcategory $\C$ of $\A$ we denote by $\res\C$ the {\em resolving closure} of $\X$, that is, the smallest resolving subcategory of $\A$ containing all the objects in $\C$.
\end{enumerate}
\end{dfn}

\begin{rem}\label{resdefrem}
The notion of a resolving subcategory has been introduced by Auslander and Bridger \cite{AB}, but we should remark that in \cite{AB} a resolving subcategory is defined to be a quasi-resolving subcategory closed under extensions.
In our sense, a resolving subcategory is also assumed to be closed under direct summands.
\end{rem}

As a trivial example, the subcategory of $\A$ consisting of objects with projective dimension less than $n$ for each $1\le n\le\infty$ is a resolving subcategory.
The catgeory $\cm(R)$ of MCM modules over a Cohen-Macaulay local ring $R$ is a resolving subcategory of $\mod R$.
There are a lot of other examples of resolving subcategories; one can find some of them in \cite[Example 2.4]{res}.

\begin{rem}\label{lmpn}
Let $\X$ be a quasi-resolving subcategory of $\A$.
\begin{enumerate}[(1)]
\item
Let $M\in\A$ and $P\in\proj\A$.
Then the equivalence
$$
M\in\X\ \Longleftrightarrow\ M\oplus P\in\X
$$
holds true.
Indeed, it is trivial that if $M$ is in $\X$, then so is $M\oplus P$.
The opposite implication follows from the exact sequence $0 \to M\xrightarrow{\binom{1}{0}} M\oplus P\xrightarrow{(0,1)}P\to0$.
\item
Let $f:M\to N$ be a morphism in $\X$.
Then one can choose a morphism $g:P\to N$ with $P\in\proj\A$ such that the morphism $(f,g):M\oplus P\to N$ is an epimorphism.
Taking the kernel, one gets a short exact sequence
$$
0 \to L \to M\oplus P \xrightarrow{(f,g)} N \to 0
$$
in $\A$.
Note that all of the three objects in this exact sequence are in $\X$.
\end{enumerate}
\end{rem}

Finally, we give a structure result on finitely presented modules over the stable category of a quasi-resolving subcategory.
The following proposition extends the result \cite[Proposition 3.3]{Y2} on module categories to abelian categories.
The first assertion is also a generalization of the result \cite[Proposition 1.1]{AR} on resolving subcategories to quasi-resolving ones.

\begin{prop}\label{3.3}
Let $\X$ be a quasi-resolving subcategory of $\A$.
\begin{enumerate}[\rm(1)]
\item
The category $\mod\underline\X$ is an abelian category with enough projective objects.
\item
For each object $F\in\mod\underline\X$ there exists an exact sequence $0 \to A \to B \to C \to 0$ in $\A$ with $A,B,C\in\X$ which induces a projective resolution of $F$ in $\mod\underline\X$:
\begin{align*}
\cdots \to \lhom_\A(-,\syz^2C)|_{\underline\X} & \to \lhom_\A(-,\syz A)|_{\underline\X} \to \lhom_\A(-,\syz B)|_{\underline\X} \to \lhom_\A(-,\syz C)|_{\underline\X} \\
&\to \lhom_\A(-,A)|_{\underline\X} \to \lhom_\A(-,B)|_{\underline\X} \to \lhom_\A(-,C)|_{\underline\X} \to F \to 0.
\end{align*}
\end{enumerate}
\end{prop}

\begin{proof}
(1) Let $f:M\to N$ be a morphism in $\X$.
There is an exact sequence $0 \to L \to M\oplus P \xrightarrow{(f,g)}N \to 0$ in $\A$ with $L\in\X$ and $P\in\proj\A$ (see Remark \ref{lmpn}).
By Lemma \ref{half}(1) this induces an exact sequence
$$
\lhom_\A(-,L)|_{\underline\X}\to\lhom_\A(-,M)|_{\underline\X}\xrightarrow{\lhom_\A(-,f)|_{\underline\X}}\lhom_\A(-,N)|_{\underline\X}
$$
in $\mod\underline\X$.
Thus $\underline\X$ has pseudokernels, and the assertion follows from Remark \ref{yoneda}(2).

(2) There is an exact sequence $\lhom_\A(-,B)|_{\underline\X} \xrightarrow{\phi} \lhom_\A(-,C)|_{\underline\X} \to F \to 0$ with $B,C\in\X$.
By Remark \ref{yoneda}(1) we find a morphism $f:B\to C$ such that $\phi=\lhom_\A(-,f)|_{\underline\X}$.
We have an exact sequence $0 \to A \to B\oplus P \xrightarrow{(f,g)} C \to 0$ with $A\in\X$ and $P\in\proj\A$ by Remark \ref{lmpn}.
Applying Proposition \ref{longex}(2) to this short exact sequence, we obtain such an exact sequence as in the assertion.
\end{proof}

\begin{rem}
We should remark that $\mod\X$ is not necessarily abelian; $\X$ does not necessarily have pseudokernels, even if it is resolving.
(For example, consider the subcategory of $\mod R$ consisting of totally reflexive modules in the case where $R$ is not Gorenstein.)
This is one of the reasons why we take the stable category $\underline\X$ of $\X$ and study $\underline\X$ rather than $\X$ itself.
\end{rem}

We recall here the definition of the singularity category of an additive category.

\begin{dfn}
Let $\C$ be an additive category such that $\mod\C$ is abelian.
The {\em singularity category} of $\C$ is by definition the Verdier quotient
$$
\ds(\C)=\db(\mod\C)/\kb(\proj(\mod\C)).
$$
\end{dfn}

Note that $\ds(\mod R)$ is nothing but the (usual) singularity category $\ds(R)$ of $R$ (see the second paragraph of Section \ref{sect1}).

Let $\X$ be a quasi-resolving subcategory of $\A$.
Then $\mod\underline\X$ is an abelian category by Proposition \ref{3.3}(1).
So one can define the singularity category $\ds(\underline\X)$ of $\underline\X$.

\section{Several conditions on quasi-resolving subcategories}\label{sect3}

In this section we compare several conditions on quasi-resolving subcategories, motivated by the work of Yoshino \cite{Y2}.
The point is to investigate the Frobenius subcategory of Gorenstein projective objects, which is invisible in \cite{Y2}, so that we extend main results of \cite{Y2}.
We also show that one of the conditions makes a certain equivalence of triangulated categories, which plays an important role in later sections.

To state those conditions, we need to recall the definitions of a cosyzygy and a Gorenstein projective object in an abelian category.

\begin{dfn}[cf. Definition \ref{25}]\label{31}
Let $M$ be an object of $\A$, and let $n$ be a positive integer.
Let $g:M\to P$ be a left ($\proj\A$)-approximation, i.e., a morphism such that any morphism from $M$ to an object in $\proj\A$ factors through $g$.
Then the cokernel of $g$ is called the {\em first cosyzygy} of $M$ and denoted by $\syz^{-1}M$.
The {\em $n$th cosyzygy} $\syz^{-n}M$ of $M$ is defined inductively as $\syz^{-1}(\syz^{-(n-1)}M)$.
For a subcategory $\X$ of $\A$ we denote by $\syz^{-n}\X$ the subcategory of $\A$ consisting of all $n$th cosyzygies.
This contains $\proj\A$.
\end{dfn}

\begin{rem}
Let $M\in\A$ and $n>0$.
Suppose that $\syz^{-n}M$ exists.
Then a complex
$$
(0 \to M \to P_{-1} \to \cdots \to P_{-n} \to \syz^{-n}M \to 0)
$$
in $\A$ with $P_{-i}\in\proj\A$ for $1\le i\le n$ is induced.
This complex is not necessarily exact.
\end{rem}

\begin{dfn}\label{33}
Let $M$ be an object of $\A$.
\begin{enumerate}[(1)]
\item
The {\em projective dimension} of $M$ in $\A$, denoted by $\pd_\A M$, is defined as the infimum of the integers $n\ge0$ such that there exists an exact sequence
$$
0 \to P_n \to P_{n-1} \to \cdots \to P_1 \to P_0 \to M \to 0
$$
in $\A$ with $P_i\in\proj\A$ for all $0\le i\le n$.
\item
An exact sequence
$$
P:\ \cdots \xrightarrow{\partial_2} P_1 \xrightarrow{\partial_1} P_0 \xrightarrow{\partial_0} P_{-1} \xrightarrow{\partial_{-1}} \cdots
$$
of projective objects of $\A$ is called a {\em complete resolution} of $M$ if $\Hom_\A(P,Q)$ is again exact for all $Q\in\proj\A$ and the image of $\partial_0$ is isomorphic to $M$.
\item
An object admitting a complete resolution is called {\em Gorenstein projective}.
The subcategory of $\A$ consisting of Gorenstein projective objects is denoted by $\gp\A$.
Since $\gp\A$ contains $\proj\A$, the stable category $\lgp\A:=\underline{\gp\A}$ is defined.
\item
The {\em Gorenstein projective dimension} of $M$ in $\A$, denoted by $\Gpd_\A M$, is defined as the infimum of the integers $n\ge0$ such that there exists an exact sequence
$$
0 \to G_n \to G_{n-1} \to \cdots \to G_1 \to G_0 \to M \to 0
$$
in $\A$ with $G_i\in\gp\A$ for all $0\le i\le n$.
\end{enumerate}
\end{dfn}

\begin{rem}\label{3.4}
\begin{enumerate}[(1)]
\item
A Gorenstein projective object of $\mod R$ is called a {\em totally reflexive module}.
The Gorenstein projective dimension of an object $M$ of $\mod R$ is called the {\em Gorenstein dimension} ({\em G-dimension} for short) of $M$, and denoted by $\gdim_RM$.
\item
Let $M$ be a Gorenstein projective object of $\A$.
Then $M$ admits the $n$th cosyzygy for all $n\ge1$, which is uniquely determined up to projective summands, and the induced complex
$$
(0 \to M \to P_{-1} \to \cdots \to P_{-n} \to \syz^{-n}M \to 0)
$$
in $\A$ with $P_{-i}\in\proj\A$ for $1\le i\le n$ is exact.
Furthermore, taking $\syz$ makes an autoequivalence of the category $\lgp\A$ and taking $\syz^{-1}$ makes its quasi-inverse.
\item
The category $\gp\A$ is a {\em Frobenius category} \cite[Chapter I, \S2]{Ha} and a resolving subcategory of $\A$, and the stable category $\lgp\A$ is a triangulated category.
These statements are shown in \cite[Proposition 2.13]{Be} except the fact that $\gp\A$ is closed under direct summands (see Remark \ref{resdefrem}).
This fact is proved by noting the following.

Suppose that there is an exact sequence $0 \to M\oplus N \xrightarrow{(f,g)} P \to L \to 0$ in $\A$ with $P\in\proj\A$ and $\Ext_\A^1(L,\proj\A)=0$.
Letting $X,Y$ be the cokernels of $f,g$ respectively, we get exact sequences $0\to M\xrightarrow{f}P\to X\to0$ and $0\to N\xrightarrow{g}P\to Y\to0$.
There is a commutative diagram
$$
\xymatrix{
&&0&0\\
0\ar[r] & M\oplus N\ar[r]^(.6){(f,g)} & P\ar[r]\ar[u] & L\ar[r]\ar[u] & 0 \\
0\ar[r] & M\oplus N\ar[r]^{\left(\begin{smallmatrix}
f&0\\
0&g
\end{smallmatrix}\right)}\ar@{=}[u] & P\oplus P\ar[r]\ar[u]_{(1,1)} & X\oplus Y\ar[r]\ar[u] & 0 \\
&&P\ar@{=}[r]\ar[u]_{\binom{1}{-1}} & P\ar[u] &\\
&&0\ar[u] & 0\ar[u] &
}
$$
with exact rows and columns.
The exact sequence $0 \to P \to X\oplus Y\to L\to0$ splits as $\Ext_\A^1(L,P)=0$, and hence $\Ext_\A^1(X,\proj\A)=0$.
Using this argument, one observes that if there is an exact sequence $0\to M\oplus N\to P^0\to P^1\to P^2\to\cdots$ whose dual by each projective is also exact, then there is an exact sequence $0\to M\to P^0\to P^0\oplus P^1\to P^0\oplus P^1\oplus P^2\to\cdots$ whose dual by each projective is also exact.
\item
Let $M\in\A$ and $n\ge0$.
Then $\Gpd_\A M\le n$ if and only if $\syz^nM\in\gp\A$.
\end{enumerate}
\end{rem}

Let us introduce regularity and Gorensteinness for a quasi-resolving subcategory.

\begin{dfn}
Let $\C$ be an additive category such that $\mod\C$ is abelian.
We say that $\C$ is {\em regular} (respectively, {\em Gorenstein}) if every object of $\mod\C$ has finite projective dimension (respectively, Gorenstein projective dimension).
A regular (respectively, Gorenstein) category $\C$ is called {\em of dimension at most $n$} if every object of $\mod\C$ has projective dimension (respectively, Gorenstein projective dimension) at most $n$, or equivalently, $\syz^n(\mod\C)$ coincides with $\proj(\mod\C)$ (respectively, $\gp(\mod\C$)).
\end{dfn}

Here are typical examples of regular and Gorenstein categories.

\begin{prop}\label{referee}
\begin{enumerate}[\rm(1)]
\item
Every abelian category is regular of dimension at most $2$.
\item
Every triangulated category is Gorenstein of dimension at most $0$.
\item
Let $\Lambda$ be an Iwanaga-Gorenstein ring of (selfinjective) dimension at most $n$.
Then the category $\proj\Lambda$ of finitely generated projective $\Lambda$-modules is Gorenstein of dimension at most $n$.
\end{enumerate}
\end{prop}

\begin{proof}
Let $\C$ be an additive category, and pick $F\in\mod\C$.
There is an exact sequence $\Hom_\C(-,M)\xrightarrow{f}\Hom_\C(-,N)\to F\to0$.
By Yoneda's lemma we find a morphism $g\in\Hom_\C(M,N)$ with $f=\Hom_\C(-,g)$.

(1) Suppose that $\C$ is abelian.
Then there is an exact sequence $0\to L\to M\xrightarrow{g}N$ in $\C$, which induces an exact sequence
$$
0\to\Hom_\C(-,L)\to\Hom_\C(-,M)\xrightarrow{f}\Hom_\C(-,N)\to F\to0
$$
in $\mod\C$.
This shows that $F$ has projective dimension at most $2$.

(2) Suppose that $\C$ is triangulated.
Then there is an exact triangle $L\to M\xrightarrow{g}N\rightsquigarrow$ in $\C$, which induces an exact sequence
$$
\cdots\to\Hom_\C(-,\sus^{-1}N)\to\Hom_\C(-,L)\to\Hom_\C(-,M)\xrightarrow{f}\Hom_\C(-,N)\to\Hom_\C(-,\sus L)\to\cdots
$$
in $\mod\C$, whose dual by each projective object of $\mod\C$ is also exact by Remark \ref{yoneda}(1) and Yoneda's lemma.
Hence this exact sequence gives a complete resolution of $F$, and so $F$ is Gorenstein projective.

(3) One has $\syz^n(\mod\Lambda)=\gp(\mod\Lambda)$ by \cite[Theorem 10.2.14]{EJ}.
The assertion follows from this; see also Lemma \ref{end} stated later.
\end{proof}

Auslander and Reiten \cite{AR4} implicitly prove the following.

\begin{thm}[Auslander-Reiten]\label{ar}
Let $\X$ be a quasi-resolving subcategory of $\A$.
Suppose that every object in $\X$ has projective dimension at most $n$ in $\A$.
Then $\underline\X$ is regular of dimension at most $3n-1$.
\end{thm}

We are interested in establishing a Gorenstein analogue of this theorem.
For this, we introduce the following condition.

\begin{dfn}
\begin{enumerate}[(1)]
\item
We say that a subcategory $\Y$ of $\gp\A$ is closed under cosyzygies provided that if $Y$ is an object in $\Y$, then so is $\syz^iY$ for all $i<0$.
\item
Let $\X$ be a quasi-resolving subcategory of $\A$, and let $n\ge0$ be an integer.
We say that $\X$ satisfies the condition $(\g_n)$ if $\syz^n\X$ is contained in $\gp\A$ and closed under cosyzygies.
\end{enumerate}
\end{dfn}

From now on we fix a quasi-resolving subcategory $\X$ of $\A$ and an integer $n\ge0$, and name several statements for convenience.
\begin{enumerate}
\item[$\ca$]
$\X$ satisfies $(\g_n)$.
\item[$\cb$]
$\underline\X$ is Gorenstein of dimension at most $3n$.
\item[$\cc$]
Every object in $\X$ has Gorenstein projective dimension at most $n$.
\item[$\cd$]
$\Ext_\A^{>n}(\X,\proj\A)=0$.
\item[$\ce$]
$\Ext_{\mod\underline\X}^{>3n}(\mod\underline\X,\proj(\mod\underline\X))=0$.
\end{enumerate}
It is not hard to observe that the implications $\ca\Rightarrow\cc\Rightarrow\cd$ and $\cb\Rightarrow\ce$ hold.
In what follows we consider other implications.
Before that, we investigate what the condition $\ca$ means.
In fact, this condition can be interpreted in several other ways.

\begin{prop}\label{3.5}
Let $\X$ be a quasi-resolving subcategory of $\A$.
The following are equivalent for $n\ge0$.
\begin{enumerate}[\rm(1)]
\item
$\X$ satisfies $(\g_n)$.
\item
$\syz^n\X$ is contained in $\gp\A$ and $\syz^i(\syz^n\X)$ is contained in $\X$ for all $i\in\Z$.
\item
Each object in $\syz^n\X$ has a complete resolution the images of whose differential morphisms are in $\X$.
\item
One has $\Ext_\A^{>n}(\X,\proj\A)=0$ and every object $M$ in $\syz^n\X$ admits an exact sequence $0\to M\to P\to M'\to0$ with $P\in\proj\A$ and $M'\in\syz^n\X$.
\item
One has $\Ext_\A^{>n}(\X,\proj\A)=0$ and there exists a subcategory $\Y$ of $\A$ with $\syz^n\X\subseteq\Y\subseteq\X$ such that each $Y\in\Y$ admits an exact sequence $0\to Y\to P\to Y'\to0$ with $P\in\proj\A$ and $Y'\in\Y$.
\end{enumerate}
\end{prop}

\begin{proof}
(1) $\Rightarrow$ (2):
The implication is obvious.

(2) $\Rightarrow$ (3):
Let $M$ be an object in $\syz^n\X$.
Then $M$ is Gorenstein projective, and admits a complete resolution $\cdots \xrightarrow{\partial_2} P_1 \xrightarrow{\partial_1} P_0 \xrightarrow{\partial_0} P_{-1} \xrightarrow{\partial_{-1}} \cdots$.
For each $i\in\Z$ the image of $\partial_i$ is isomorphic to $\syz^iM$ (up to projective summands), which is in $\syz^i(\syz^n\X)$ and so in $\X$.

(3) $\Rightarrow$ (4):
Let $M\in\syz^n\X$.
Then $M$ has a complete resolution $\cdots \xrightarrow{\partial_2} P_1 \xrightarrow{\partial_1} P_0 \xrightarrow{\partial_0} P_{-1} \xrightarrow{\partial_{-1}} \cdots$ such that $\Im\partial_i$ is in $\X$ for all $i\in\Z$.
Hence $\Ext_\A^{>0}(M,\proj\A)=0$, and there is an exact sequence $0 \to M \to P_{-1} \to N \to 0$ with $N:=\Im\partial_{-1}$.
Note that $N=\syz^n(\Im\partial_{-n-1})$, which belongs to $\syz^n\X$.

(4) $\Rightarrow$ (5):
The implication follows by setting $\Y:=\syz^n\X$.

(5) $\Rightarrow$ (1):
Let $M$ be an object in $\syz^n\X$.
Take a projective resolution
\begin{equation}\label{syzygies}
\cdots \xrightarrow{\partial_2} P_1 \xrightarrow{\partial_1} P_0 \to M \to 0
\end{equation}
of $M$.
It is seen that $\Im\partial_i\in\syz^n\X$ for all $i>0$.
As $\Ext_\A^{>0}(\syz^n\X,\proj\A)=0$, the $Q$-dual of \eqref{syzygies} is again exact for all $Q\in\proj\A$.
Since $M$ is in $\Y$, so there is an exact sequence $0 \to M \to P_{-1} \to M' \to 0$ with $P_{-1}\in\proj\A$ and $M'\in\Y$.
Hence there is an exact sequence $0 \to M' \to P_{-2} \to M'' \to 0$ with $P_{-2}\in\proj\A$ and $M''\in\Y$.
Iterating this procedure gives rise to an exact sequence
\begin{equation}\label{cosyzygies}
0 \to M \to P_{-1} \xrightarrow{\partial_{-1}} P_{-2} \xrightarrow{\partial_{-2}} \cdots
\end{equation}
with $P_i\in\proj\A$ and $\Im\partial_i\in\Y$ for all $i<0$.
Since $\Im\partial_i=\syz^n(\Im\partial_{i-n})\in\syz^n\X$ and $\Ext_\A^{>0}(\syz^n\X,\proj\A)=0$, the $Q$-dual of \eqref{cosyzygies} is an exact sequence for all $Q\in\proj\A$.
Splicing \eqref{syzygies} and \eqref{cosyzygies} yields a complete resolution of $M$, which shows that $M$ is Gorenstein projective, and $\syz^iM=\Im\partial_i$ belongs to $\syz^n\X$ for all $i\in\Z$.
\end{proof}

The following result shows that $\cd$ implies $\ce$.

\begin{prop}\label{qf}
Let $n$ be a nonnegative integer, and let $\X$ be a quasi-resolving subcategory of $\A$.
If $\Ext_\A^{>n}(\X,\proj\A)=0$, then $\Ext_{\mod\underline\X}^{>3n}(\mod\underline\X,\proj(\mod\underline\X))=0$.
\end{prop}

\begin{proof}
Let $F$ be an object in $\mod\underline\X$.
By virtue of Proposition \ref{3.3} there is an exact sequence
\begin{equation}\label{210}
0 \to X_{2} \to X_{1} \to X_{0} \to 0
\end{equation}
in $\A$ with $X_0,X_1,X_2\in\X$ which induces a projective resolution
$$
\cdots \to \lhom_\A(-,\syz X_{0})|_{\underline\X} \to \lhom_\A(-,X_{2})|_{\underline\X} \to \lhom_\A(-,X_{1})|_{\underline\X} \to \lhom_\A(-,X_{0})|_{\underline\X} \to F \to 0
$$
of $F$ in $\mod\underline\X$.
In view of Remark \ref{yoneda}(1), we have only to show that $\Ext_{\mod\underline\X}^{3n+1}(F,\lhom_\A(-,Y)|_{\underline\X})=0$ for $Y\in\X$.
The induced complex
\begin{align*}
\big(\Hom_{\mod\underline\X}(\lhom_\A(-,\syz^nX_0)|_{\underline\X},\lhom_\A(-,Y)|_{\underline\X})&\to\Hom_{\mod\underline\X}(\lhom_\A(-,\syz^nX_1)|_{\underline\X},\lhom_\A(-,Y)|_{\underline\X})\\
&\to\Hom_{\mod\underline\X}(\lhom_\A(-,\syz^nX_2)|_{\underline\X},\lhom_\A(-,Y)|_{\underline\X})\big)
\end{align*}
is isomorphic to the induced complex
$$
\big(\lhom_\A(\syz^nX_0,Y)\to\lhom_\A(\syz^nX_1,Y)\to\lhom_\A(\syz^nX_2,Y)\big)
$$
by Remark \ref{yoneda}(1) again.
This comes from the short exact sequence $0\to\syz^nX_2\to\syz^nX_1\to\syz^nX_0\to0$ obtained by applying $\syz^n$ to \eqref{210}.
Since $\Ext_\A^1(\syz^nX_0,\proj\A)=0$, it is exact by Lemma \ref{half}(2).
\end{proof}

The main result of this section is the following theorem, which is a Gorenstein analogue of Theorem \ref{ar} due to Auslander and Reiten.
It not only shows that the implication $\ca\Rightarrow\cb$ holds, but also establishes an equivalence of triangulated categories which analyzes the strucuture of singularity categories.

\begin{thm}\label{main3}
Let $n$ be a nonnegative integer.
Let $\X$ be a quasi-resolving subcategory of $\A$ satisfying $(\g_n)$.
Then $\underline\X$ is Gorenstein of dimension at most $3n$, and there is a triangle equivalence
$$
\ds(\underline\X)\cong\lgp(\mod\underline\X).
$$
\end{thm}

\begin{proof}
The second assertion follows from the first assertion and \cite[Corollary 4.13]{Be}.
So let us show the first assertion.
It is obvious that $\gp(\mod\underline\X$) is contained in $\syz^{3n}(\mod\underline\X$).
Let $F$ be an object in $\mod\underline\X$.
Proposition \ref{3.3} yields a short exact sequence
\begin{equation}\label{ses}
0 \to A \to B \to C \to 0
\end{equation}
in $\A$ with $A,B,C\in\X$ which induces a projective resolution
\begin{align*}
\cdots \to \lhom_\A(-,\syz^2C)|_{\underline\X} & \to \lhom_\A(-,\syz A)|_{\underline\X} \to \lhom_\A(-,\syz B)|_{\underline\X} \to \lhom_\A(-,\syz C)|_{\underline\X} \\
&\to \lhom_\A(-,A)|_{\underline\X} \to \lhom_\A(-,B)|_{\underline\X} \to \lhom_\A(-,C)|_{\underline\X} \to F \to 0
\end{align*}
of $F$ in $\mod\underline\X$.
Then $G:=\Ker(\lhom_\A(-,\syz^nA)|_{\underline\X} \to \lhom_\A(-,\syz^nB)|_{\underline\X})$ is in $\syz^{3n}(\mod\underline\X)$.
We want to show that $G$ is Gorenstein projective.
Applying $\syz^n$ to \eqref{ses}, we have an exact sequence
$$
0 \to L \xrightarrow{f} M \xrightarrow{g} N \to 0,
$$
where $L=\syz^nA$, $M=\syz^nB$ and $N=\syz^nC$.
Applying Proposition \ref{longex}(1) to the induced short exact sequence $0\to\syz^{-i}L\to\syz^{-i}M\to\syz^{-i}N\to0$ of objects in $\syz^n\X$ for each $i>0$ (or making a dual argument to the proof of Proposition \ref{longex}(1)), we obtain a sequence
$$
[\cdots\to\syz^2N\to\syz L\to\syz M\to\syz N\to L\to M\to N\to\syz^{-1}L\to\syz^{-1}M\to\syz^{-1}N\to\syz^{-2}L\to\cdots]
$$
of morphisms in $\A$ with $\syz^iC\in\X$ for all $i\in\Z$ and $C\in\{L,M,N\}$ such that any consective two morphisms become a short exact sequence by adding some projective object.
Hence we get an exact sequence
\begin{align}\label{compres}
\cdots&\to\lhom_\A(-,\syz^2N)|_{\underline\X}\to\lhom_\A(-,\syz L)|_{\underline\X}\to\lhom_\A(-,\syz M)|_{\underline\X}\to\lhom_\A(-,\syz N)|_{\underline\X}\\
&\to\lhom_\A(-,L)|_{\underline\X}\to\lhom_\A(-,M)|_{\underline\X}\to\lhom_\A(-,N)|_{\underline\X}\to\lhom_\A(-,\syz^{-1}L)|_{\underline\X}\notag\\
&\to\lhom_\A(-,\syz^{-1}M)|_{\underline\X}\to\lhom_\A(-,\syz^{-1}N)|_{\underline\X}\to\lhom_\A(-,\syz^{-2}L)|_{\underline\X}\to\cdots\notag
\end{align}
of projective objects in $\mod\underline\X$.
Applying $\lhom_\A(-,Y)|_{\underline\X}$ with $Y\in\X$, we have an exact sequence
$$
\cdots\to\lhom_\A(\syz^{-1}L,Y)\to\lhom_\A(N,Y)\to\lhom_\A(M,Y)\to\lhom_\A(L,Y)\to\lhom_\A(\syz N,Y)\to\cdots;
$$
see Remark \ref{yoneda}(1) and Lemma \ref{half}(2).
Therefore \eqref{compres} gives a complete resolution of $G$, and thus $G$ is a Gorenstein projective object of $\mod\underline\X$.
\end{proof}

In the rest of this section, we give several applications of our results.
First, applying Proposition \ref{qf} and Theorem \ref{main3} for $n=0$ yields the following corollary, which recovers \cite[Theorems 3.5 and 3.7]{Y2}.

\begin{cor}[Yoshino]
Let $\X$ be a quasi-resolving subcategory of $\A$.
\begin{enumerate}[\rm(1)]
\item
If $\Ext_\A^{>0}(\X,\proj\A)=0$, then $\mod\underline\X$ is a quasi-Frobenius category.
\item
Suppose that $\X$ satisfies $(\g_0)$.
Then one has $\mod\underline\X=\gp(\mod\underline\X)$.
In particular, $\mod\underline\X$ is a Frobenius category, whose stable category is triangle equivalent to $\ds(\underline\X)$.
\end{enumerate}
\end{cor}

Recall that a {\em thick subcategory} of a triangulated category is by definition a triangulated subcategory closed under direct summands.

\begin{rem}
The above corollary especially says that the following statements hold for a Gorenstein local ring $R$, where $\lcm(R):=\underline{\cm(R)}$ is the stable category of $\cm(R)$.
\begin{enumerate}[(1)]
\item
Let $\X$ be a resolving subcategory of $\mod R$ contained in $\cm(R)$.
Then $\mod\underline\X$ is a quasi-Frobenius category.
\item
Let $\C$ be a thick subcategory of $\lcm(R)$.
Then $\mod\C$ is a Frobenius category.
\end{enumerate}
\end{rem}

For each $n\ge0$, denote by $\Gpd_n(\A)$ the subcategory of $\A$ consisting of objects having Gorenstein projective dimension at most $n$.
Using Remark \ref{3.4} and Theorem \ref{main3}, we have the following result.

\begin{cor}\label{3.14}
\begin{enumerate}[\rm(1)]
\item
The subcategory $\Gpd_n(\A)$ of $\A$ is resolving and satisfies $\syz^n(\Gpd_n(\A))=\gp\A$.
\item
The subcategory $\Gpd_n(\A)$ of $\A$ satisfies $(\g_n)$.
\item
The stable category $\underline{\Gpd_n(\A)}$ is Gorenstein of dimension at most $3n$.
\item
There is a triangle equivalence $\ds(\underline{\Gpd_n(\A)})\cong\lgp(\mod\underline{\Gpd_n(\A)})$.
\end{enumerate}
\end{cor}

Next, we apply our theorem to complete intersections.
Let $R$ be a local ring with maximal ideal $\m$.
We say that $R$ is a {\em complete intersection} if its $\m$-adic completion is the residue ring of a regular local ring by an ideal generated by a regular sequence.
The {\em complete intersection dimension} ({\em CI-dimension} for short) of an $R$-module $M$, denoted by $\cidim_RM$, is defined as the infimum of the quantities
$$
\pd_S (M\otimes_RR')-\pd_SR',
$$
where $R \to R' \leftarrow S$ runs over the quasi-deformations of $R$.
Here, a diagram $R \overset{f}{\to} R' \overset{g}{\leftarrow} S$ of homomorphisms of local rings is called a {\em quasi-deformation} of $R$ if $f$ is faithfully flat and $g$ is a surjection whose kernel is generated by an $S$-sequence.
The notion of CI-dimension has been introduced by Avramov, Gasharov and Peeva \cite{AGP} to which we refer the reader for details of CI-dimension.

\begin{cor}\label{cia}
Let $R$ be a local ring.
\begin{enumerate}[\rm(1)]
\item
Let $n$ be a nonnegative integer.
Let $\X$ be a resolving subcategory of $\mod R$ all of whose objects have CI-dimension at most $n$.
Then $\X$ satisfies $(\g_n)$.
Hence $\underline\X$ is Gorenstein of dimension at most $3n$, and there is a triangle equivalence $\ds(\underline\X)\cong\lgp(\mod\underline\X)$.
\item
Let $R$ be a complete intersection of dimension $d$.
Let $\X$ be a resolving subcategory of $\mod R$.
Then $\X$ satisfies $(\g_d)$.
Hence $\underline\X$ is Gorenstein of dimension at most $3d$, and there is a triangle equivalence $\ds(\underline\X)\cong\lgp(\mod\underline\X)$.
\end{enumerate}
\end{cor}

\begin{proof}
(1) We have $\gdim_RM\le\cidim_RM\le n$ for all $R$-modules $M$ by \cite[Theorem (1.4)]{AGP}.
According to Remark \ref{3.4}(1)(2), the subcategory $\syz^n\X$ of $\mod R$ is contained in $\gp(\mod R)$.
Since every module in $\syz^n\X$ has CI-dimension at most $0$ by \cite[Lemma (1.9)]{AGP}, it follows from \cite[Theorem 4.15]{radius} that
$$
\syz^{-1}(\syz^n\X)\subseteq\res\syz^n\X\subseteq\X
$$
holds.
In particular, every module $M\in\syz^{-1}\syz^n\X$ has CI-dimension at most $n$, and $\cidim_RM=\gdim_RM\le0$ by \cite[Theorem (1.4)]{AGP} and the fact that $\gp(\mod R)$ is closed under cosyzygies.
Using \cite[Theorem 4.15]{radius} again, we obtain
$$
\syz^{-2}(\syz^n\X)=\syz^{-1}(\syz^{-1}\syz^n\X)\subseteq\res\syz^{-1}\syz^n\X\subseteq\X.
$$
Iterating this procedure shows that the second condition in Proposition \ref{3.5} is satisfied for $\A=\mod R$, and hence $\X$ satisfies $(\g_n)$.
Theorem \ref{main3} complete the proof of the corollary.

(2) Taking advantage of \cite[Theorems (1.3) and (1.4)]{AGP}, one observes that every $R$-module has CI-dimension at most $d$.
The assertion thus follows from (1).
\end{proof}

\section{Vanishing of Ext of modules over stable resolving subcategories}\label{sect4}

In this section we consider when the converse of Proposition \ref{qf} holds, in other words, when $\ce$ implies $\cd$.
We focus on the category of finitely generated modules over a commutative noetherian local ring which is Henselian (e.g. complete).
Throughout this section, let $R$ be a Henselian local ring with maximal ideal $\m$ and residue field $k$.

We say that two $R$-modules $M$ and $N$ are {\em stably isomorphic} if $M\cong N$ in the stable category $\lmod R:=\underline{\mod R}$ of $\mod R$, or equivalently, if $M\oplus P\cong N\oplus Q$ in $\mod R$ for some free $R$-modules $P,Q$.
Also, for an $R$-module $X$ we denote the {\em (Auslander) transpose} of $X$ by $\tr X$.
This is defined to be the cokernel of the map $\Hom_R(\phi,R)$, where $\phi$ appears in an $R$-free presentation $F_1\xrightarrow{\phi}F_0\to X\to0$ of $X$.
This is uniquely determined up to free summands, so uniquely determined up to isomorphism in the stable category $\lmod R$.

Let $\SS_n$ (respectively, $\T_n$) be the subcategory of $\mod R$ consisting of modules $M$ with $\Ext_R^i(M,R)=0$ (respectively, $\Ext_R^i(\tr M,R)=0$) for all $1\le i\le n$.
A module in $\T_n$ is called an {\em $n$-torsionfree} module.
There is an intimate relationship between these categories; see \cite[Proposition 1.1.1 and Corollary 1.1.2]{I}.

\begin{lem}\label{iyam}
For all $n,m\ge0$ there is a diagram of functors
$$
\xymatrix@R-1pc@C-1pc{
\underline{\SS_{n+1}\cap\T_m} \ar@<.5mm>[rr]^{\syz^{\phantom{-1}}} \ar@<.5mm>[dd]^{\tr} & & \underline{\SS_n\cap\T_{m+1}} \ar@<.5mm>[ll]^{\syz^{-1}} \ar@<.5mm>[dd]^{\tr}\\
\\
\underline{\SS_m\cap\T_{n+1}} \ar@<.5mm>[rr]^{\syz^{-1}} \ar@<.5mm>[uu]^{\tr} & & \underline{\SS_{m+1}\cap\T_n} \ar@<.5mm>[ll]^{\syz^{\phantom{-1}}} \ar@<.5mm>[uu]^{\tr}
}
$$
which commutes up to isomorphism, whose horizontal arrows are equivalences and whose vertical arrows are dualities.
\end{lem}

The following is the main result of this section, which asserts that $\ce$ implies $\cd$ in a certain setting.

\begin{thm}\label{main4}
Let $n\ge0$ be an integer, and let $\X$ be a resolving subcategory of $\mod R$ satisfying the following three conditions.
\begin{enumerate}[\rm(1)]
\item
The $R$-modules of projective dimension less than $n$ are in $\X$.
\item
For each $R$-sequence $\xx=x_1,\dots,x_n$ the module $R/\xx R$ belongs to $\X$.
\item
For all $X\in\X$ the module $\syz^{n+1}X$ is $(n+1)$-torsionfree.
\end{enumerate}
If $\Ext_{\mod\underline\X}^{3n+1}(\mod\underline\X,\proj(\mod\underline\X))=0$, then $\Ext_R^{>n}(\X,R)=0$.
\end{thm}

This theorem is proved in five steps.
The first four steps are done by the following four lemmas, respectively.

\begin{lem}\label{ess1}
Let $n$ be a nonnegative integer.
Let $\X$ be a quasi-resolving subcategory of $\mod R$ with $\Ext_{\mod\underline\X}^{3n+1}(\mod\underline\X,\proj(\mod\underline\X))=0$.
Let $0 \to Y \to E \to X \to 0$ be an exact sequence of modules in $\X$ with $\pd_RY\le n$ and $\pd_RE>n$ such that $X$ is an indecomposable $R$-module belonging to $\SS_n$.
Then $\syz^nE$ is stably isomorphic to $\syz^nX$.
\end{lem}

\begin{proof}
Taking $n$th syzygies, we have an exact sequence $0 \to \syz^nY \to \syz^nE\oplus P \xrightarrow{(f,r)} \syz^nX \to 0$, where $P$ is a free $R$-module.
(This follows from the horseshoe and snake lemmas.)
Since $Y$ has projective dimension at most $n$, the $R$-module $\syz^nY$ is free.
Proposition \ref{longex}(2) gives rise to an exact sequence
\begin{align*}
0=\lhom_R(-,\syz^nY)|_{\underline\X}\to\lhom_R(-,\syz^nE)|_{\underline\X}&\xrightarrow{\alpha}\lhom_R(-,\syz^nX)|_{\underline\X}\to\lhom_R(-,\syz^{n-1}Y)|_{\underline\X}\\
&\to\cdots\to\lhom_R(-,E)|_{\underline\X}\to\lhom_R(-,X)|_{\underline\X}\to F\to0.
\end{align*}
Then $F$ belongs to $\mod\underline\X$.
As $\syz^nE$ is in $\X$, the assumption implies $\Ext_{\mod\underline\X}^{3n+1}(F,\lhom_R(-,\syz^nE)|_{\underline\X})=0$.
Therefore the map
$$
\Hom_{\mod\underline\X}(\lhom_R(-,\syz^nX)|_{\underline\X},\lhom_R(-,\syz^nE)|_{\underline\X})\to\Hom_{\mod\underline\X}(\lhom_R(-,\syz^nE)|_{\underline\X},\lhom_R(-,\syz^nE)|_{\underline\X})
$$
induced by $\alpha$ is surjective, which shows that the map $\lhom_R(f,\syz^nE):\lhom_R(\syz^nX,\syz^nE)\to\lhom_R(\syz^nE,\syz^nE)$ is also surjective (see Remark \ref{yoneda}(1)).
Hence there exists a homomorphism $g:\syz^nX\to\syz^nE$ of $R$-modules such that $\underline{1}=\underline{gf}$ in $\lhom_R(\syz^nE,\syz^nE)$.
Therefore $\syz^nE$ is stably isomorphic to a direct summand of $\syz^nX$, and there is an isomorphism $\syz^nX\cong\syz^nE\oplus C$ in $\underline{\T_n}$.
Sending this isomorphism by the $n$th cosyzygy functor $\syz^{-n}$, we obtain an isomorphism
$$
X\cong\syz^{-n}\syz^nE\oplus\syz^{-n}C
$$
in $\underline{\SS_n}$ by Lemma \ref{iyam}.
In view of the assumption that $E$ has projective dimension more than $n$, we observe that $X$ is a nonzero indecomposable object and $\syz^{-n}\syz^nE$ is a nonzero object in $\lmod R$.
As $\lmod R$ is a Krull-Schmidt category, $X$ is stably isomorphic to $\syz^{-n}\syz^nE$, which yields isomorphisms $\syz^nX\cong\syz^n(\syz^{-n}\syz^nE)\cong\syz^nE$ in $\lmod R$ by Lemma \ref{iyam}.
\end{proof}

\begin{lem}\label{ess2}
Let $n\ge0$ be an integer.
Let $\X$ be a quasi-resolving subcategory of $\mod R$ closed under direct summands such that $\Ext_{\mod\underline\X}^{3n+1}(\mod\underline\X,\proj(\mod\underline\X))=0$.
Let $X$ be an $R$-module in $\X\cap\SS_n$ with $\pd_RX=n+1$.
Then the $R$-module $\tr\syz^nX$ is stably isomorphic to a $k$-vector space.
\end{lem}

\begin{proof}
Let $X=X_1\oplus\cdots\oplus X_a\oplus X_{a+1}\oplus\cdots\oplus X_b$ be a direct sum decomposition of $X$ into indecomposable modules with $\pd_RX_i=n+1$ for $1\le i\le a$ and $\pd_R X_i\le n$ for $a+1\le i\le b$.
Since $\X$ is closed under direct summands, $X_i$ is in $\X\cap\SS_n$ for all $i$, and $\tr\syz^nX\cong\tr\syz^nX_1\oplus\cdots\oplus\tr\syz^nX_a$ in $\lmod R$.
Thus we may assume that $X$ is an indecomposable $R$-module.

Let $\varepsilon:0\to\syz X\to F\xrightarrow{\pi}X\to0$ be an exact sequence with $F$ free.
Fix an element $a\in\m$.
There is a pullback diagram
$$
\begin{CD}
\phantom{a}\varepsilon:\quad @. 0 @>>> \syz X @>>> F @>{\pi}>> X @>>> 0\phantom{,}\\
@. @. @| @AAA @AAaA \\
a\varepsilon:\quad @. 0 @>>> \syz X @>>> Y @>>> X @>>> 0,
\end{CD}
$$
and it is seen that $Y$ is in $\X$.
Using the exact sequence
\begin{equation}\label{yfxx}
0 \to Y \to F\oplus X \xrightarrow{(\pi,a)} X \to 0
\end{equation}
and the fact that $a$ is an element in $\m$, we easily observe that $Y$ has projective dimension $n+1$.
Hence we can apply Lemma \ref{ess1} to the short exact sequence $a\varepsilon$ to get an isomorphism
\begin{equation}\label{snsn}
\syz^nY\cong\syz^nX\text{ in }\lmod R.
\end{equation}
Applying $\Ext_R^*(-,R)$ to \eqref{yfxx} induces an exact sequence
$$
\Ext_R^{n+1}(X,R)\xrightarrow{a}\Ext_R^{n+1}(X,R)\xrightarrow{\beta}\Ext_R^{n+1}(Y,R)\to\Ext_R^{n+2}(X,R).
$$
Since $X$ has projective dimension $n+1$, the module $\Ext_R^{n+2}(X,R)$ vanishes, which implies that the map $\beta$ is surjective.
Using \eqref{snsn}, we have isomorphisms
$$
\Ext_R^{n+1}(X,R)\cong\Ext_R^1(\syz^nX,R)\cong\Ext_R^1(\syz^nY,R)\cong\Ext_R^{n+1}(Y,R).
$$
According to \cite[Theorem 2.4]{M}, the map $\beta$ is an isomorphism, and hence $a\Ext_R^{n+1}(X,R)=0$.
It follows that $\m\Ext_R^{n+1}(X,R)=0$, which means that $\Ext_R^1(\syz^nX,R)=\Ext_R^{n+1}(X,R)$ is a $k$-vector space.
As $\syz^nX$ has projective dimension $1$, there is an exact sequence
$$
0\to P_1\to P_0\to\syz^nX\to0
$$
with $P_1,P_0$ free.
It is seen from this that $\Ext_R^1(\syz^nX,R)$ is stably isomorphic to $\tr\syz^nX$.
Consequently, the $R$-module $\tr\syz^nX$ is stably isomorphic to a $k$-vector space.
\end{proof}

\begin{lem}\label{ess3}
Let $n$ be a nonnegative integer.
Let $\X$ be a resolving subcategory of $\mod R$ satisfying $\Ext_{\mod\underline\X}^{3n+1}(\mod\underline\X,\proj(\mod\underline\X))=0$.
Then $\X\cap\SS_n$ contains no module of projective dimension $n+1$.
\end{lem}

\begin{proof}
Suppose that the subcategory $\X\cap\SS_n$ contains a module $X$ of projective dimension $n+1$.
Then it follows from Lemma \ref{ess2} that $\tr\syz^nX$ is stably isomorphic to $k^{\oplus t}$ for some $t\ge0$.
As the projective dimension of $X$ exceeds $n$, the integer $t$ has to be positive.
There are isomorphisms
$$
X\cong\tr\syz^n\tr\syz^nX\cong\tr\syz^nk^{\oplus t}
$$
in $\lmod R$, where the first isomorphism follows from Lemma \ref{iyam}.
Since $t$ is positive, $X$ is in $\X\cap\SS_n$ and $\X$ is closed under direct summands, the module $\tr\syz^nk$ is in $\X\cap\SS_n$.

Let $e$ be the embedding dimension of $R$; note $e\ge1$.
Write $\m=(x_1,x_2,\dots,x_e)$, and take the subideal $I=(x_1^2,x_2,\dots,x_e)$.
Then $\m/I$ is isomorphic to $k$, which means that there is an exact sequence
$$
0 \to k \to R/I \to k \to 0.
$$
As $\depth R\ge\pd_RX=n+1$, we have $\Ext_R^n(k,R)=0$.
By \cite[Lemma 2.3]{crspd}, applying $\tr\syz^n$ induces an exact sequence
$$
0 \to \tr\syz^nk \to \tr\syz^n(R/I) \to \tr\syz^nk \to 0.
$$
Since the subcategory $\X\cap\SS_n$ of $\mod R$ is closed under extensions, $\tr\syz^n(R/I)$ belongs to $\X\cap\SS_n$.
It follows from \cite[Lemma 2.5]{crspd} that the $R$-module $\tr\syz^n(R/I)$ has projective dimension $n+1$.
Thus one can apply Lemma \ref{ess2} to this module to see that $\tr\syz^n(\tr\syz^n(R/I))$ is stably isomorphic to $k^{\oplus r}$ for some $r\ge0$.
As $R/I$ has finite length and $\depth R\ge n+1$, we see that $R/I$ is in $\SS_n$.
Lemma \ref{iyam} implies that $\tr\syz^n(\tr\syz^n(R/I))$ is stably isomorphic to $R/I$.
Hence $R/I$ is stably isomorphic to $k^{\oplus r}$.
Using the Krull-Schmidt theorem, we observe that $R/I$ is isomorphic to $k$ as an $R$-module.
Taking the annihilators yields $I=\m$, which is a contradiction.
We conclude that $\X\cap\SS_n$ does not contain a module of projective dimension $n+1$.
\end{proof}

\begin{lem}\label{ess4}
Let $n$ be a nonnegative integer.
Let $\X$ be a resolving subcategory of $\mod R$ such that $\Ext_{\mod\underline\X}^{3n+1}(\mod\underline\X,\proj(\mod\underline\X))=0$.
If $R/\xx R$ belongs to $\X$ for all $R$-sequences $\xx=x_1,\dots,x_n$ and $\syz^{n+1}(\X\cap\SS_n)\subseteq\T_{n+1}$, then one has $\Ext_R^{n+1}(\X\cap\SS_n,R)=0$.
\end{lem}

\begin{proof}
Let $X$ be a module in $\X\cap\SS_n$.
We want to show that $\Ext_R^{n+1}(X,R)=0$.
Since $\X\cap\SS_n$ is closed under direct summands, we may assume that $X$ is indecomposable.
According to Lemma \ref{ess3}, the $R$-module $X$ does not have projective dimension $n+1$.
If $\pd_RX\le n$, then it obviously holds that $\Ext_R^{n+1}(X,R)=0$.
So we may assume that $\pd_RX\ge n+2$.
The module $\syz^{n+1}X$ belongs to $\T_{n+1}$ by assumption, and we have $\grade\Ext_R^{n+1}(X,R)\ge n$ by \cite[Proposition (2.26) and Corollary (4.18)]{AB}.

Suppose that $\Ext_R^{n+1}(X,R)\ne0$.
Then the annihilator $\ann_R\Ext_R^{n+1}(X,R)$ is a proper ideal of $R$, and one can take an $R$-sequence $\xx=x_1,\dots,x_n$ in it.
Letting $K=\K(\xx,R)$ be the Koszul complex of $R$ with respect to $\xx$, one gets a free resolution
$$
0 \to K_n \xrightarrow{\delta} K_{n-1} \to \cdots \to K_1 \to K_0 \to R/\xx R \to 0
$$
of the $R$-module $R/\xx R$.
For an $R$-module $M$ we denote by $F^i(M)$ the composition
\begin{align*}
\Ext_R^i(M,R/\xx R)&\to\Ext_R^{i+1}(M,\syz(R/\xx R))\to\Ext_R^{i+2}(M,\syz^2(R/\xx R))\\
&\to\cdots\to\Ext_R^{i+n-1}(M,\syz^{n-1}(R/\xx R))\to\Ext_R^{i+n}(M,K_n)=\Ext_R^{i+n}(M,R)
\end{align*}
of connecting homomorphisms.

We claim that $F^1(X):\Ext_R^1(X,R/\xx R)\to\Ext_R^{n+1}(X,R)$ is an isomorphism.
In fact, since $X$ is in $\SS_n$, the composition $f:\Ext_R^1(X,R/\xx R)\to\Ext_R^2(X,\syz(R/\xx R))\to\cdots\to\Ext_R^n(X,\syz^{n-1}(R/\xx R))$ is bijective and the map $g:\Ext_R^n(X,\syz^{n-1}(R/\xx R))\to\Ext_R^{n+1}(X,R)$ is injective.
Since $\xx$ annihilates $\Ext_R^{n+1}(X,R)$, the map $\Ext_R^{n+1}(X,\delta):\Ext_R^{n+1}(X,K_n)\to\Ext_R^{n+1}(X,K_{n-1})$ is a zero map, and we see that the injective map $g$ is surjective.
Therefore $F^1(X)=gf$ is an isomorphism.

As we assume that $\Ext_R^{n+1}(X,R)$ does not vanish, neither does $\Ext_R^1(X,R/\xx R)$, and there exists a nonsplit short exact sequence
$$
\sigma:0 \to R/\xx R \xrightarrow{\theta} E \to X \to 0
$$
of $R$-modules.
By assumption $R/\xx R$ is in $\X$, and $\X$ is closed under extensions.
Hence $E$ is also in $\X$.
One has $\pd_R(R/\xx R)=n$, and $\pd_RE\ge n+2$ since $\pd_RX\ge n+2$.
Therefore we can apply Lemma \ref{ess1} to see that $\syz^nE$ is stably isomorphic to $\syz^nX$.
There is a commutative diagram
$$
\begin{CD}
\Hom_R(R/\xx R,R/\xx R) @>\alpha>> \Ext_R^1(X,R/\xx R) @>>> \Ext_R^1(E,R/\xx R) @>>> \Ext_R^1(R/\xx R,R/\xx R)\\
@V{F^0(R/\xx R)}VV @V{F^1(X)}V{\cong}V @V{F^1(E)}VV @V{F^1(R/\xx R)}VV \\                                                                  
\Ext_R^n(R/\xx R,R) @>\beta>> \Ext_R^{n+1}(X,R) @>\gamma>> \Ext_R^{n+1}(E,R) @>>> \Ext_R^{n+1}(R/\xx R,R)
\end{CD}
$$
with exact rows.
The fact that $R/\xx R$ has projective dimension $n$ implies $\Ext_R^{n+1}(R/\xx R,R)=0$, which shows that the map $\gamma$ is surjective.
As $\syz^nX$ is stably isomorphic to $\syz^nE$, the module $\Ext_R^{n+1}(X,R)$ is isomorphic to $\Ext_R^{n+1}(E,R)$.
It is seen from \cite[Theorem 2.4]{M} that $\gamma$ is an isomorphism, and hence $\beta=0$.
Since $F^1(X)$ is an isomorphism, diagram chasing implies that $\alpha=0$.
Hence the map $\Hom_R(\theta,R/\xx R):\Hom_R(E,R/\xx R)\to\Hom_R(R/\xx R,R/\xx R)$ is surjective, which means that the exact sequence $\sigma$ splits.
This contradiction shows that $\Ext_R^{n+1}(X,R)=0$.
\end{proof}

Now we have reached (the last fifth step of) the proof of the theorem.

\begin{proof}[Proof of Theorem \ref{main4}]
Since $\X$ is closed under syzygies, it suffices to deduce that $\Ext_R^{n+1}(X,R)=0$ for each $R$-module $X$ in $\X$.
The assumption (3) implies that $\syz^nX$ is $n$-torsionfree; see \cite[Corollary (4.18)]{AB}.
It follows from \cite[Proposition (2.21)]{AB} that there is an exact sequence
$$
0 \to Y \to Z \to X \to 0
$$
of $R$-modules with $Z=\tr\syz^n\tr\syz^nX$ such that $Y$ has projective dimension less than $n$.
The assumption (1) shows that $Y$ is in $\X$, and hence so is $Z$.
As $\syz^nX$ is in $\T_n$, Lemma \ref{iyam} implies that $Z$ is in $\SS_n$.
Thus we have $Z\in\X\cap\SS_n$.
Thanks to the assumptions (2) and (3), one can apply Lemma \ref{ess4} to see that $\Ext_R^{n+1}(Z,R)=0$.
Since $\Ext_R^n(Y,R)=0$, the above short exact sequence shows that $\Ext_R^{n+1}(X,R)=0$, which is what we want.
\end{proof}

From now on to the end of this section, we give several applications of Theorem \ref{main4}.
The first one is the following corollary, where the assumption on $R$ is satisfied for instance when $R$ is a Cohen-Macaulay local ring with an isolated singularity.
(Recall that a local ring $R$ is said to have an {\em isolated singularity} if for each nonmaximal prime ideal $\p$ of $R$ the local ring $R_\p$ is regular.)

\begin{cor}\label{m4c}
Let $R$ be a $d$-dimensional Cohen-Macaulay local ring that is locally Gorenstein on the punctured spectrum.
Let $\X$ be a resolving subcategory of $\mod R$ containing the modules of finite projective dimension.
Then $\Ext_R^{>d}(\X,R)=0$ if and only if $\Ext_{\mod\underline\X}^{>3d}(\mod\underline\X,\proj(\mod\underline\X))=0$.
\end{cor}

\begin{proof}
The `only if' part follows from Proposition \ref{qf}.
As to the `if' part, in view of Theorem \ref{main4}, it is enough to prove that $\syz^{d+1}\X$ is contained in $\T_{d+1}$.
We have
$$
\syz^{d+1}\X\subseteq\syz^{d+1}(\mod R)=\syz\cm(R)\subseteq\T_{d+1}.
$$
Indeed, the first inclusion is obvious.
It follows from \cite[Theorem 3.8]{EG} that $\cm(R)=\T_d=\syz^d(\mod R)$, which implies the equality.
As for the second inclusion, let $M$ be an MCM $R$-module.
Then there is an exact sequence
$$
0 \to F \to N \to M \to 0
$$
of $R$-modules such that $\Ext_R^1(N,R)$ vanishes and $F$ is free; see \cite[Proposition (2.21)]{AB}.
This short exact sequence shows that $N$ is an MCM $R$-module, and we have $N\in\SS_1\cap\T_d$.
Lemma \ref{iyam} implies that $\syz M=\syz N$ is in $\T_{d+1}$, and thus the inclusion considered follows.
\end{proof}

The next application is a characterization of Gorenstein rings.
Note that the second condition in the result below corresponds to $\cb$ for $n=d$.

\begin{cor}\label{4.5}
Let $R$ be a $d$-dimensional Cohen-Macaulay local ring which is locally Gorenstein on the punctured spectrum.
The following are equivalent.
\begin{enumerate}[\rm(1)]
\item
$R$ is Gorenstein.
\item
$\lmod R$ is Gorenstein of dimension at most $3d$.
\item
$\Ext_{\mod(\lmod R)}^{3d+1}(\mod(\lmod R),\proj(\mod(\lmod R)))=0$.
\end{enumerate}
\end{cor}

\begin{proof}
(1) $\Rightarrow$ (2):
Since $R$ is Gorenstein, one has $\syz^d(\mod R)=\cm(R)=\gp(\mod R)$, and this is closed under cosyzygies.
Theorem \ref{main3} shows the implication.

(2) $\Rightarrow$ (3):
The implication is straightforward.

(3) $\Rightarrow$ (1):
By virtue of Corollary \ref{m4c}, one obtains $\Ext_R^{>d}(\mod R,R)=0$, which especially says that $\Ext_R^{>d}(k,R)=0$.
Therefore $R$ is a Gorenstein ring.
\end{proof}

Let $n=0$ in Theorem \ref{main4}.
Then the three conditions (1)--(3) in the theorem are trivially satisfied; see \cite[Lemma 3.4]{EG}.
Hence the following result holds, whose two assertions are nothing but \cite[Theorem 4.2]{Y2} and \cite[Corollary 4.3]{Y2}, respectively.
(The second assertion is shown along the same lines as in the proof of Corollary \ref{4.5}.)

\begin{cor}
\begin{enumerate}[\rm(1)]
\item
Let $\X$ be a resolving subcategory of $\mod R$.
If $\mod\underline\X$ is a quasi-Frobenius category, then $\Ext_R^{>0}(\X,R)=0$.
\item
The following statements are equivalent.
\begin{enumerate}[\rm(a)]
\item
$R$ is an artinian Gorenstein ring.
\item
$\mod(\lmod R)$ is a Frobenius category.
\item
$\mod(\lmod R)$ is a quasi-Frobenius category.
\end{enumerate}
\end{enumerate}
\end{cor}

\section{A sufficient condition for singular equivalence}\label{sect5}

In this section we study singular equivalences among stable quasi-resolving subcategories.
First of all, let us make the precise definition of singular equivalence.

\begin{dfn}
Let $\C$ and $\C'$ be additive categories such that $\mod\C$ and $\mod\C'$ are abelian.
We say that $\C$ and $\C'$ are {\em singularly equivalent} if there exists a triangle equivalence $\ds(\C)\cong\ds(\C')$.
\end{dfn}

Our main interest is to ask when this is the case for the stable categories of quasi-resolving subcategories.
The main result of this section gives an answer to this question.
We begin with a proposition.

\begin{prop}\label{ho}
Let $\X$ and $\Y$ be quasi-resolving subcategories of $\A$.
Let $n$ be a nonnegative integer.
The following are equivalent.
\begin{enumerate}[\rm(1)]
\item
One has $\syz^n\X\subseteq\Y\subseteq\X\cap\gp\A$ and $\Y$ is closed under cosyzygies.
\item
One has $\syz^n\X\subseteq\Y\subseteq\X\cap\gp\A$ and $\syz^n\X$ is closed under cosyzygies.
\item
One has $\syz^n\X=\Y=\X\cap\gp\A$ and this is closed under cosyzygies.
\end{enumerate}
When this is the case, $\X\cap\gp\A$ is closed under extensions.
\end{prop}

\begin{proof}
It is trivial that (3) implies (1), and we see from Proposition \ref{3.5} that (1) implies (2).
Pick an object $M\in\X\cap\gp\A$.
Then $M$ is Gorenstein projective, so we have $M\cong\syz^{-n}(\syz^nM)$.
As $M$ is in $\X$, the syzygy $\syz^nM$ is in $\syz^n\X$.
If $\syz^n\X$ is closed under cosyzygies, then $\syz^{-n}(\syz^nM)$ belongs to $\syz^n\X$ and so does $M$.
This shows that (2) implies (3), and consequently, the conditions (1)--(3) are equivalent.

Suppose that one of the three equivalent conditions is satisfied.
Let $0\to L\xrightarrow{f}M\to N\to0$ be an exact sequence in $\A$ with $L,N\in\X\cap\gp\A$.
There is an exact sequence $0\to L\xrightarrow{g}P\to\syz^{-1}L\to0$ with $P\in\proj\A$.
As $\Ext_\A^1(N,P)=0$, the pushout diagram of $f$ and $g$ gives rise to an exact sequence
$$
0\to M\to P\oplus N\to\syz^{-1}L\to0.
$$
The object $P\oplus N$ is in $\X$, and so is $\syz^{-1}L$ since $\X\cap\gp\A$ is closed under cosyzygies.
By definition $\X$ is closed under kernels of epimorphisms, and hence $M$ belongs to $\X$.
As $\gp\A$ is closed under extensions, $M$ is in $\X\cap\gp\A$.
Thus $\X\cap\gp\A$ is closed under extensions.
\end{proof}

As an immediate consequence of Proposition \ref{ho}, we have:

\begin{cor}
Let $\X$ be a quasi-resolving subcategory of $\A$.
If $\X$ satisfies $(\g_n)$. then $\syz^n\X=\X\cap\gp\A$.
\end{cor}

Now we state the main result of this section.

\begin{thm}\label{main1}
Let $\X$ be a quasi-resolving subcategory of $\A$ satisfying the condition $(\g_n)$ for some $n\ge0$.
\begin{enumerate}[\rm(1)]
\item
The equality $\syz^n\X=\X\cap\gp\A$ holds.
Denote this subcategory by $\Y$.
\item
One has the following.
\begin{enumerate}[\rm(a)]
\item
$\Y$ is a quasi-resolving subcategory of $\A$ satisfying $(\g_0)$.
\item
$\Y$ is a Frobenius subcategory with $\proj\Y=\proj\A$.
\item
$\underline\Y$ is a triangulated category.
\item
$\underline\Y$ is a strictly full subcategory of $\lgp\A$.
\end{enumerate}
\item
There are triangle equivalences
$$
\ds(\underline\X)\cong\lgp(\mod\underline\X)=\underline{\syz^{3n}(\mod\underline\X)}\cong\underline{\mod\underline\Y}=\lgp(\mod\underline\Y)\cong\ds(\underline\Y),
$$
where the middle equivalence is induced by the restriction functor $F\mapsto F|_{\underline\Y}$.
In particular, $\underline\X$ and $\underline\Y$ are singularly equivalent.
\end{enumerate}
\end{thm}

\begin{proof}
(1) The assertion follows from Propositions \ref{ho} and \ref{3.5}.

(2) Proposition \ref{ho} implies that $\Y$ is closed under cosyzygies and extensions.
Since $\X$ and $\gp\A$ are quasi-resolving, so is $\Y$.
Thus (a) and (b) hold, while (c) is a consequence of (b).
The induced functor $\underline\Y\to\lgp\A$ is fully faithful, and via this functor we can regard $\underline\Y$ as a full subcategory of $\lgp\A$.
In view of Remark \ref{lmpn}(1), we easily observe that $\underline\Y$ is closed under isomorphism as a full subcategory of $\lgp\A$.
This shows (d).

(3) Note that $\Y$ satisfies $(\g_0)$.
By virtue of Theorem \ref{main3}, we have equalities $\gp(\mod\underline\X)=\syz^{3n}\mod\underline\X$ and $\gp(\mod\underline\Y)=\mod\underline\Y$, and triangle equivalences $\ds(\underline\X)\cong\lgp(\mod\underline\X)$ and $\ds(\underline\Y)\cong\lgp(\mod\underline\Y)$.
The restriction $F\mapsto F|_{\underline\Y}$ makes a covariant exact functor
$$
\Phi:\Mod\underline\X\to\Mod\underline\Y
$$
of abelian categories.
We establish several claims.

\begin{claim}\label{c1}
The functor $\Phi$ sends $\lhom_\A(-,X)|_{\underline\X}$ with $X\in\X$ to $\lhom_\A(-,Y)|_{\underline\Y}$ with $Y=\syz^{-n}\syz^nX\in\Y$.
Hence one has $\Phi(\proj(\mod\underline\X))\subseteq\proj(\mod\underline\Y)$.
\end{claim}

\begin{proof}[Proof of Claim]
Using Proposition \ref{longex}(2) and the fact that $\Ext_{\A}^{>0}(\Y,\proj\A)=0$, we have isomorphisms
\begin{align*}
\lhom_{\A}(-,X)|_{\underline\Y}&\cong \Ext_{\A}^1(-,\syz X)|_{\underline\Y} 
\cong \Ext_\A^2(-,\syz^2X)|_{\underline\Y}\cong \cdots \cong \Ext_{\A}^n(-,\syz^n X)|_{\underline\Y},\\
\lhom_{\A}(-,\syz^{-n} \syz^n X)|_{\underline\Y}&\cong \Ext_{\A}^1(-,\syz^{1-n} \syz^n X)|_{\underline\Y}\cong \Ext_{\A}^2(-,\syz^{2-n} \syz^n X)|_{\underline\Y}\cong\cdots\cong\Ext_\A^n(-,\syz^nX)|_{\underline\Y}.
\end{align*}
Therefore there is an isomorphism $\lhom_{\A}(-,X)|_{\underline\Y}\cong\lhom_{\A}(-,\syz^{-n} \syz^n X)|_{\underline\Y}$.
Since $\syz^nX\in\Y$ and $\syz^{-1}\Y\subseteq\Y$, we have $\syz^{-n} \syz^n X\in\Y$.
\renewcommand{\qedsymbol}{$\square$}
\end{proof}

It is easy to see from Claim \ref{c1} that $\Phi$ induces an exact functor $\mod\underline\X\to\mod\underline\Y$ and a triangle functor
$$
\phi: \underline{\syz^{3n}\mod\underline\X} \longrightarrow \underline{\mod\underline\Y}.
$$

\begin{claim}\label{c2}
The functor $\phi$ is dense.
\end{claim}

\begin{proof}[Proof of Claim]
Let $G$ be an object in $\mod\underline\Y$.
Take a projective presentation $\lhom_{\A}(-,Y_1)|_{\underline\Y} \xrightarrow{p} \lhom_{\A}(-,Y_0)|_{\underline\Y} \to G \to 0$ with $Y_0, Y_1\in\Y$.
One can write $p=\lhom_{\A}(-,f)|_{\underline\Y}$ for some morphism $f:Y_1 \to Y_0$ (see Remark \ref{yoneda}(1)). 
Let $F$ be the cokernel of the morphism
$$
\lhom_{\A}(-,\syz^{-n}f)|_{\underline\X}:\lhom_\A(-,\syz^{-n}Y_1)|_{\underline\X}\to\lhom_\A(-,\syz^{-n}Y_0)|_{\underline\X}.
$$
As $\syz^{-n}Y_0$ and $\syz^{-n}Y_1$ are in $\Y$ and hence in $\X$, it is seen that $\lhom_{\A}(-,\syz^{-n}f)|_{\underline\X}$ is an $\underline\X$-homomorphism of projective $\underline\X$-modules, and $F$ belongs to $\mod\underline\X$.
There is an exact sequence $0 \to Y_2 \to Y_1\oplus P \xrightarrow{(f,g)} Y_0 \to 0$ with $P\in\proj\A$ and $Y_2\in\Y$ (see Remark \ref{lmpn}), which induces an exact sequence $0 \to \syz^{-n}Y_2 \to \syz^{-n}Y_1 \to \syz^{-n}Y_0 \to 0$.
Proposition \ref{longex}(2) shows that there is an exact sequence
\begin{align*}
\lhom_\A(-,Y_1)|_{\underline\X}&\xrightarrow{\lhom_\A(-,f)|_{\underline\X}}\lhom_\A(-,Y_0)|_{\underline\X}\to\lhom_\A(-,\syz^{-1}Y_2)|_{\underline\X}\to\cdots\\
&\to\lhom_\A(-,\syz^{-n}Y_1)|_{\underline\X}\xrightarrow{\lhom_\A(-,\syz^{-n}f)|_{\underline\X}}\lhom_\A(-,\syz^{-n}Y_0)|_{\underline\X}\to F\to0.
\end{align*}
Letting $H$ be the cokernel of $\lhom_\A(-,f)|_{\underline\X}$, we observe that $H$ is the $(3n)$th syzygy of $F$.
Thus $H$ is in $\syz^{3n}\mod\underline\X$, and it is obvious that $\phi(H)=G$.
Consequently, $\phi$ is a dense functor.
\renewcommand{\qedsymbol}{$\square$}
\end{proof}

\begin{claim}\label{c3}
The functor $\phi$ is full.
\end{claim}

\begin{proof}[Proof of Claim]
Take objects $F,G\in\syz^{3n}\mod\underline\X$ and a morphism $\rho\in\Hom_{\mod\underline\Y}(F|_{\underline\Y},G|_{\underline\Y})$.
There exists an object $F'\in\mod\underline\X$ whose $(3n)$th syzygy is $F$.
Using Proposition \ref{3.3}(2), we obtain an exact sequence $0 \to X_2 \to X_1 \to X_0 \to 0$ of objects in $\X$ which induces a projective resolution
\begin{align*}
\lhom_\A(-,\syz^nX_1)|_{\underline\X} &\xrightarrow{\lhom_\A(-,f)|_{\underline\X}} \lhom_\A(-,\syz^nX_0)|_{\underline\X} \to \lhom_\A(-,\syz^{n-1}X_2)|_{\underline\X}\\
&\to \cdots \to \lhom_\A(-,X_2)|_{\underline\X} \to  \lhom_\A(-,X_1)|_{\underline\X} \to \lhom_\A(-,X_0)|_{\underline\X} \to F' \to 0
\end{align*}
of $F'$ in $\mod\underline\X$, where $f:\syz^nX_1\to\syz^nX_0$ is a morphism in $\A$.
By Schanuel's lemma, we may assume that $F$ is the cokernel of $\lhom_\A(-,f)|_{\underline\X}$.
Applying the same argument to $G$, we get two exact sequences
\begin{align*}
&\lhom_\A(-,\syz^nX_1)|_{\underline\X} \xrightarrow{\lhom_\A(-,f)|_{\underline\X}} \lhom_\A(-,\syz^nX_0)|_{\underline\X} \to F \to 0,\\
&\lhom_\A(-,\syz^nX_1')|_{\underline\X} \xrightarrow{\lhom_\A(-,g)|_{\underline\X}} \lhom_\A(-,\syz^nX_0')|_{\underline\X} \to G \to 0
\end{align*}
with $X_0',X_1'\in\X$ and $g\in\Hom_\A(\syz^nX_1',\syz^nX_0')$.
As $\syz^nX_i$ and $\syz^nX_i'$ are in $\Y$ for $i=0,1$, the objects $\lhom_\A(-,\syz^nX_i)|_{\underline\Y}$ and $\lhom_\A(-,\syz^nX_i')|_{\underline\Y}$ are projective in $\mod\underline\Y$.
Hence there is a commutative diagram
$$
\begin{CD}
\lhom_\A(-,\syz^nX_1)|_{\underline\Y} @>{\lhom_\A(-,f)|_{\underline\Y}}>> \lhom_\A(-,\syz^nX_0)|_{\underline\Y} @>>> F|_{\underline\Y} @>>> 0\phantom{,}\\
@V{\lhom_\A(-,h_1)|_{\underline\Y}}VV @V{\lhom_\A(-,h_0)|_{\underline\Y}}VV @V{\rho}VV \\
\lhom_\A(-,\syz^nX_1')|_{\underline\Y} @>{\lhom_\A(-,g)|_{\underline\Y}}>> \lhom_\A(-,\syz^nX_0')|_{\underline\Y} @>>> G|_{\underline\Y} @>>> 0,
\end{CD}
$$
where $h_i:\syz^nX_i\to\syz^nX_i'$ is a morphism in $\A$ for $i=0,1$.
Substituting $\syz^nX_1$ for ``$-$'', we observe that $\underline{h_0f}=\underline{gh_1}$.
Hence the square in the diagram below commutes.
$$
\begin{CD}
\lhom_\A(-,\syz^nX_1)|_{\underline\X} @>{\lhom_\A(-,f)|_{\underline\X}}>> \lhom_\A(-,\syz^nX_0)|_{\underline\X} @>>> F @>>> 0\phantom{.}\\
@V{\lhom_\A(-,h_1)|_{\underline\X}}VV @V{\lhom_\A(-,h_0)|_{\underline\X}}VV \\
\lhom_\A(-,\syz^nX_1')|_{\underline\X} @>{\lhom_\A(-,g)|_{\underline\X}}>> \lhom_\A(-,\syz^nX_0')|_{\underline\X} @>>> G @>>> 0.
\end{CD}
$$
Let $\xi:F\to G$ be the induced morphism by this diagram.
Then $\xi$ is in $\Hom_{\mod\underline\X}(F,G)$ and we have $\phi(\underline\xi)=\underline{\xi|_{\underline\Y}}=\underline{\rho}$.
This shows that $\phi$ is a full functor.
\renewcommand{\qedsymbol}{$\square$}
\end{proof}

\begin{claim}\label{c4}
The functor $\phi$ is faithful.
\end{claim}

\begin{proof}[Proof of Claim]
Let $\xi:F\to G$ be a morphism in $\mod\underline\X$ with $F,G\in\syz^{3n}\mod\underline\X$.
Suppose that $\underline\xi$ is sent to $\underline0$ by the functor $\phi$.
Then $\underline{\xi|_{\underline\Y}}=\underline0$, which says that the morphism $\xi|_{\underline\Y}:F|_{\underline\Y}\to G|_{\underline\Y}$ factors through some $P\in\proj(\mod\underline\Y)$.
As in the proof of Claim \ref{c3}, we may assume that $F,G$ have projective presentations
\begin{align*}
&\lhom_\A(-,Y_1)|_{\underline\X} \xrightarrow{\lhom_\A(-,f)|_{\underline\X}} \lhom_\A(-,Y_0)|_{\underline\X} \to F \to 0,\\
&\lhom_\A(-,Y_1')|_{\underline\X} \xrightarrow{\lhom_\A(-,g)|_{\underline\X}} \lhom_\A(-,Y_0')|_{\underline\X} \to G \to 0
\end{align*}
with $Y_i,Y_i'\in\Y$ and $f\in\Hom_\A(Y_1,\Y_0),\,g\in\Hom_\A(Y_1',Y_0')$, and there is a commutative diagram
$$
\begin{CD}
\lhom_\A(-,Y_1)|_{\underline\X} @>{\lhom_\A(-,f)|_{\underline\X}}>> \lhom_\A(-,Y_0)|_{\underline\X} @>>> F @>>> 0\phantom{,}\\
@V{\lhom_\A(-,h_1)|_{\underline\X}}VV @V{\lhom_\A(-,h_0)|_{\underline\X}}VV @V{\xi}VV\\
\lhom_\A(-,Y_1')|_{\underline\X} @>{\lhom_\A(-,g)|_{\underline\X}}>> \lhom_\A(-,Y_0')|_{\underline\X} @>>> G @>>> 0,
\end{CD}
$$
where $h_i\in\Hom_\A(Y_i,Y_i')$ for $i=0,1$.
We also have an equality $\underline{h_0f}=\underline{gh_1}$.
Sending the above diagram by the functor $\phi$, we obtain the following diagram in $\mod\underline\Y$:
$$
\xymatrix@C45pt{
\lhom_\A(-,Y_1)|_{\underline\Y}\ar[rr]^{\lhom_\A(-,f)|_{\underline\Y}}\ar[dd]_{\lhom_\A(-,h_1)|_{\underline\Y}} & & \lhom_\A(-,Y_0)|_{\underline\Y}\ar[rr]^(0.6)\gamma\ar[dd]_{\lhom_\A(-,h_0)|_{\underline\Y}}\ar@{-->}[lldd]_{\lhom_\A(-,\underline\ell)|_{\underline\Y}\qquad} & & F|_{\underline\Y}\ar[r]\ar[dd]_{\xi|_{\underline\Y}}\ar[ld]_(0.55)\alpha & 0\\
& & & P\ar@{-->}[ld]_(0.4)\varepsilon\ar[rd]^(0.4)\beta\\
\lhom_\A(-,Y_1')|_{\underline\Y}\ar[rr]^{\lhom_\A(-,g)|_{\underline\Y}} & & \lhom_\A(-,Y_0')|_{\underline\Y}\ar[rr]^(0.6)\delta & & G|_{\underline\Y}\ar[r] & 0
}
$$
Since $P$ is a projective object of $\mod\underline\Y$, the morphism $\beta$ factors through $\delta$; we find a morphism $\varepsilon:P\to\lhom_\A(-,Y_0')|_{\underline\Y}$ with $\beta=\delta\varepsilon$.
Hence there exists a morphism $\lhom_\A(-,\underline\ell)|_{\underline\Y}:\lhom_\A(-,Y_0)|_{\underline\Y}\to\lhom_\A(-,Y_1')|_{\underline\Y}$ with $\ell\in\Hom_\A(Y_0,Y_1')$ such that
$$
\lhom_\A(-,h_0)|_{\underline\Y}=\lhom_\A(-,g)|_{\underline\Y}\cdot\lhom_\A(-,\ell)|_{\underline\Y}+\varepsilon\alpha\gamma.
$$
Composition with $\lhom_\A(-,f)|_{\underline\Y}$ shows that $\lhom_\A(-,h_0f)|_{\underline\Y}=\lhom_\A(-,g\ell f)|_{\underline\Y}$, which implies $\underline{h_0f}=\underline{g\ell f}$.
The left triangle in the diagram
$$
\xymatrix@C45pt{
\lhom_\A(-,Y_1)|_{\underline\X}\ar[rr]^{\lhom_\A(-,f)|_{\underline\X}}\ar[rrdd]_(0.4)0 & & \lhom_\A(-,Y_0)|_{\underline\X}\ar[rr]\ar[dd]_{\lhom_\A(-,h_0-g\ell)|_{\underline\X}} & & F\ar[r]\ar[dd]_{\xi}\ar@{-->}[lldd]^\zeta & 0\\
 \\
\lhom_\A(-,Y_1')|_{\underline\X}\ar[rr]^{\lhom_\A(-,g)|_{\underline\X}} & & \lhom_\A(-,Y_0')|_{\underline\X}\ar[rr] & & G\ar[r] & 0
}
$$
commutes, which implies that there exists a morphism $\zeta:F\to\lhom_\A(-,Y_0')|_{\underline\X}$ such that the right upper triangle commutes, and so does the right lower one.
It follows that the morphism $\xi$ factors through the projective object $\lhom_\A(-,Y_0')|_{\underline\X}$ of $\mod\underline\X$, and we have $\underline\xi=\underline0$.
The functor $\phi$ is thus faithful.
\renewcommand{\qedsymbol}{$\square$}
\end{proof}

Combining Claims \ref{c2}, \ref{c3} and \ref{c4} implies that the triangle functor $\phi$ is an equivalence, and the proof of the theorem is completed.
\end{proof}

\begin{rem}\label{oi}
The assumption of Theorem \ref{main1} does not necessarily imply that there is an equivalence
$$
\db(\mod\underline\X)\cong\db(\mod\underline\Y).
$$
Indeed, let $R$ be a regular local ring of positive Krull dimension $d$.
Set $n=d$ and $\A=\X=\mod R$.
Then $\gp\A=\cm(R)$ consists of free $R$-modules, and the assumption of Theorem \ref{main1} is satisfied.
There are embeddings $\underline\X\hookrightarrow\mod\underline\X\hookrightarrow\db(\mod\underline\X)$, and since $\underline\X\ne\zero$, we have $\db(\mod\underline\X)\ne\zero$.
On the other hand, since $\underline\Y=\zero$, we have $\db(\mod\underline\Y)=\zero$.
Hence $\db(\mod\underline\X)\ncong\db(\mod\underline\Y)$.
\end{rem}

The remainder of this section is devoted to stating immediate applications of our Theorem \ref{main1}.
First we apply the theorem to a quasi-resolving subcategory containing the Gorenstein projective objects.

\begin{cor}\label{containgp}
\begin{enumerate}[\rm(1)]
\item
Let $\X$ be a quasi-resolving subcategory of $\A$ with $\syz^n\X\subseteq\gp\A\subseteq\X$ for some $n\ge0$.
Then $\underline\X$ is singularly equivalent to $\lgp\A$.
\item
One has $\underline{\Gpd_n(\A)}$ is singularly equivalent to $\lgp\A$ for each $n\ge0$.
In particular, $\underline{\Gpd_i(\A)}$ and $\underline{\Gpd_j(\A)}$ are singularly equivalent for all $i,j\ge0$.
\end{enumerate}
\end{cor}

\begin{proof}
(1) Since $\gp\A$ is closed under cosyzygies, Proposition \ref{3.5} guarantees that $\X$ satisfies $(\g_n)$.
Theorem \ref{main1}(3) implies the assertion.

(2) The assertion follows by applying (1) to $\X=\Gpd_n(\A)$, or by directly combining Corollary \ref{3.14}(2) with Theorem \ref{main1}(3).
\end{proof}

Applying the above result to the module category of a Gorenstein ring, one observes that all resolving subcategories containing the MCM modules are singularly equivalent:

\begin{cor}\label{containcm}
Let $R$ be a Gorenstein local ring.
Let $\X$ be a quasi-resolving subcategory of $\mod R$ containing $\cm(R)$.
Then $\underline\X$ and $\lcm(R)$ are singularly equivalent.
\end{cor}

\begin{proof}
Putting $d=\dim R$, we have $\syz^d\X\subseteq\cm(R)\subseteq\X$.
Since $\cm(R)=\gp(\mod R)$, the assertion follows from Corollary \ref{containgp}(1).
\end{proof}

Here we give an example of a singular equivalence which is analogous to Corollary \ref{containcm}.

\begin{ex}
Let $R$ be a $d$-dimensional Gorenstein local ring.
Let $\mod_0(R)$ stand for the category of $R$-modules that are locally free on the punctured spectrum of $R$, and set $\cm_0(R)=\cm(R)\cap\mod_0(R)$.
Let $\X$ be a quasi-resolving subcategory of $\mod R$ contained in $\mod_0(R)$ and containing $\cm_0(R)$.
Then one has $\syz^d\X\subseteq\cm_0(R)\subseteq\X$ and $\cm_0(R)$ is closed under cosyzygies.
Proposition \ref{3.5} implies that $\X$ satisfies $(\g_n)$, and therefore $\underline\X$ is singularly equivalent to $\lcm_0(R)$ by Theorem \ref{main1}.
\end{ex}

\begin{rem}\label{krause}
In general, the existence of singular equivalences between $\underline\X$ and $\underline\Y$ for quasi-resolving subcategories $\X,\Y$ does not imply that there is an inclusion relation between $\X$ and $\Y$, as is shown by the example below.

Let $R$ be a Gorenstein local integral domain of Krull dimension at least $2$ with unique maximal ideal $\m$.
Let $\X$ (respectively, $\Y$) be the subcategory of $\mod R$ consisting of modules $M$ satisfying $\m\notin\ass_RM$ (respectively, $\ass_RM\subseteq\{0,\m\}$), where $\ass_RM$ stands for the set of associated prime ideals of $R$.
Then it is easy to observe that $\X$ and $\Y$ are resolving subcategories of $\mod R$ containing $\cm(R)$.
In fact, in the bijection constructed in \cite[Corollary 8.9]{Sa} the resolving subcategory $\X$ (respectively, $\Y$) corresponds to the {\em grade-consistent function} $f$ (respectively, $g$) defined as follows.
$$
f(\p)=
\begin{cases}
\height\p & \text{if }\p\in\spec R\setminus\{\m\},\\
\height\p-1 & \text{if }\p=\m,
\end{cases}
\qquad
g(\p)=
\begin{cases}
\height\p & \text{if }\p\in\{0,\m\},\\
\height\p-1 & \text{if }\p\in\spec R\setminus\{0,\m\}.
\end{cases}
$$
It follows from Corollary \ref{containcm} that $\underline\X$ and $\underline\Y$ are singularly equivalent; they are also singularly equivalent to $\lcm(R)$.
On the other hand, the module $R/\m$ is in $\Y$ but not in $\X$, and the module $R/\p$ is in $\X$ but not in $\Y$ for any prime ideal $\p$ of $R$ different from $0$ and $\m$ (such a prime ideal exists since $R$ has Krull dimension at least $2$).
Therefore, $\X$ and $\Y$ have no inclusion relation.
\end{rem}

From now on, we apply our theorem to complete intersection local rings.

\begin{cor}\label{sgci}
Let $R$ be a local ring.
Let $\X$ be a subcategory of $\mod R$.
\begin{enumerate}[\rm(1)]
\item
Assume that $\X$ is resolving and that all modules in $\X$ have CI-dimension at most $n$.
Then $\syz^n\X=\X\cap\gp(\mod R)$ holds, and $\syz^n\X$ is resolving.
One has that $\underline\X$ is singularly equivalent to $\underline{\syz^n\X}$.
\item
If $R$ is a complete intersection, then $\underline{\res\syz^i\X}$ and $\underline{\res\syz^j\X}$ are singularly equivalent for all $i,j\ge0$.
\end{enumerate}
\end{cor}

\begin{proof}
(1) The category $\gp(\mod R)$ consists of totally reflexive $R$-modules, which is a resolving subcategory of $\mod R$.
Hence $\X\cap\gp(\mod R)$ is also a resolving subcategory of $\mod R$.
Corollary \ref{cia}(1) and Theorem \ref{main1} yield $\syz^n\X=\X\cap\gp(\mod R)$, and its stable category is singularly equivalent to $\underline\X$.

(2) It suffices to prove that $\underline{\res\syz^i\X}$ is singularly equivalent to $\underline{\res\syz^d\X}$ for all $i\ge0$, where $d:=\dim R$.
By \cite[Theorems (1.3) and (1.4)]{AGP} and (1) we observe that $\syz^d(\res\syz^i\X)$ is resolving, and $\underline{\res\syz^i\X}$ is singularly equivalent to $\underline{\syz^d(\res\syz^i\X)}$.
It is enough to verify that
\begin{equation}\label{571}
\syz^d(\res\syz^i\X)=\res\syz^d\X.
\end{equation}
For this, we establish a claim.
\begin{claim*}
For any subcategory $\C$ of $\mod R$ and any integer $n\ge0$ one has the inclusion
$$
\syz^n(\res\C)\subseteq\res\syz^n\C.
$$
\end{claim*}
\begin{proof}[Proof of Claim]
Let $\M$ be the subcategory of $\mod R$ consisting of modules whose $n$th syzygies are in $\res\syz^n\C$.
Then it is easily seen that $\M$ is a resolving subcategory containing $\C$.
Hence $\M$ contains $\res\C$, which deduces the claim.
\renewcommand{\qedsymbol}{$\square$}
\end{proof}
Let us prove the equality \eqref{571}.
Since $\syz^i\X$ is contained in $\res\X$, so is $\res\syz^i\X$.
Hence $\syz^d(\res\syz^i\X)$ is contained in $\syz^d(\res\X)$, which is contained in $\res\syz^d\X$ by the claim.
On the other hand, since $\syz^d\X$ is contained in $\cm(R)$, it is observed by \cite[Corollary 4.16]{radius} that $\syz^{-i}\syz^d\X$ is contained in $\res\syz^d\X$.
Hence $\syz^d\X=\syz^i(\syz^{-i}\syz^d\X)$ is contained in $\syz^i(\res\syz^d\X)$, which is contained in $\res\syz^{d+i}\X$ by the claim.
Thus $\res\syz^d\X$ is contained in $\res\syz^{d+i}\X=\res\syz^d(\syz^i\X)$, which is contained in $\res\syz^d(\res\syz^i\X)=\syz^d(\res\syz^i\X)$.
This implies that $\syz^d(\res\syz^i\X)$ contains $\res\syz^d\X$.
\end{proof}

\begin{rem}
It is not known that finiteness of CI-dimension is preserved by taking extensions; see \cite[Remark 5.1]{NS} for example.
If this turns out to be true, then the second assertion of Corollary \ref{sgci} will extend to arbitrary local rings $R$ and subcategories $\X$ all of whose objects have finite CI-dimension, because all modules in $\res\syz^i\X$ will have finite CI-dimension.
(The same proof will work by some appropriate replacement; $d=\dim R$ should be replaced with $\depth R$, and so on.)
\end{rem}

The following result immediately follows from Corollary \ref{sgci}(1).
This is regarded as a stronger version of Corollary \ref{containcm} for resolving subcategories.
Thanks to this result, to classify singular equivalence classes over a complete intersection, one has only to consider resolving subcategories consisting of MCM modules.

\begin{cor}\label{xxcm}
Let $R$ be a local complete intersection.
Let $\X$ be a resolving subcategory of $\mod R$.
Then $\underline\X$ is singularly equivalent to $\underline{\X\cap\cm(R)}$.
\end{cor}

Let $R$ be a local ring with maximal ideal $\m$.
Recall that $R$ is called a {\em hypersurface} if the $\m$-adic completion of $R$ is isomorphic to the quotient ring of a regular local ring by a principal ideal.
Needless to say, any hypersurface is a complete intersection.
Over a hypersurface with an isolated singularity, there are at most two singular equivalence classes.

\begin{cor}\label{either}
Let $R$ be a local hypersurface with an isolated singularity.
Let $\X$ be a resolving subcategory of $\mod R$.
Then $\underline\X$ is singularly equivalent to either $\lcm(R)$ or $\zero$.
\end{cor}

\begin{proof}
It follows from Corollary \ref{xxcm} that $\underline\X$ is singularly equivalent to $\underline{\X\cap\cm(R)}$.
As $\X\cap\cm(R)$ is a resolving subcategory of $\mod R$ contained in $\cm(R)$, it is equal to either $\proj(\mod R)$ or $\cm(R)$ by \cite[Corollary 6.9(1)]{stcm}.
Hence $\underline{\X\cap\cm(R)}$ coincides with either $\zero$ or $\lcm(R)$.
\end{proof}

\section{Regularity for resolving subcategories and simple singularities}\label{sect6}

We have learned in Corollary \ref{either} that over an isolated hypersurface singularity the stable category $\underline\X$ of a resolving subcategory $\X$ is singularly equivalent to either $\lcm(R)$ or $\zero$.
So it is natural to ask when $\underline\X$ is singularly equivalent to $\zero$.
The main purpose of this section is to give an answer to this question.

The {\em global dimension} (respectively, {\em finitistic dimension}) of $\A$ is defined to be the supremum of the projective dimensions (respectively, finite projective dimensions) of objects of $\A$.
Let us begin with investigating the relationships of the condition that $\underline\X$ is singularly equivalent to $\zero$ with several other conditions, including finiteness of the global dimension of $\mod\underline\X$.

\begin{prop}\label{sgzfgl}
Let $\X$ be a quasi-resolving subcategory of $\A$.
Consider the following four conditions.
\begin{enumerate}[\rm(1)]
\item
$\mod\underline\X$ has finite global dimension.
\item
$\underline\X$ is singularly equivalent to $\zero$.
\item
$\underline\X$ is regular.
\item
Every object of $\X$ has finite projective dimension in $\A$.
\end{enumerate}
Then the implications $(1)\Rightarrow(2)\Leftrightarrow(3)\Leftarrow(4)$ hold.
The implication $(1)\Leftarrow(2)$ holds if $\X$ satisfies the condition $(\g_n)$.
The implication $(1)\Leftarrow(4)$ holds if the finitistic dimension of $\A$ is finite.
\end{prop}

\begin{proof}
First of all, note that the condition (2) is equivalent to the equality $\db(\mod\underline\X)=\kb(\proj(\mod\underline\X))$.

$(2)\Rightarrow(3)$:
It is straightforward to verify the implication.

$(3)\Rightarrow(2)$:
We observe that $\mod\underline\X$ is contained in $\kb(\proj(\mod\underline\X))$ as subcategories of $\db(\mod\underline\X)$.
Let $C=(0\to C^a\to C^{a+1}\to\cdots\to C^b\to0)$ be any object of $\db(\mod\underline\X)$.
Then $C$ belongs to the smallest thick subcategory $\T$ of $\db(\mod\underline\X)$ containing $M:=C^a\oplus C^{a+1}\oplus\cdots\oplus C^b\in\mod\underline\X$.
Since $\kb(\proj(\mod\underline\X))$ is a thick subcategory of $\db(\mod\underline\X)$ containing $M$, we see that $\kb(\proj(\mod\underline\X))$ contains $\T$.
Hence $C$ belongs to $\kb(\proj(\mod\underline\X))$.
Thus $\db(\mod\underline\X)$ coincides with $\kb(\proj(\mod\underline\X))$.

$(1)\Rightarrow(3)$:
The implication is trivial.

$(4)\Rightarrow(3)$:
Let $F$ be an object in $\mod\underline\X$.
Using Proposition \ref{3.3}(2), we have an exact sequence $0 \to A \to B \to C \to 0$ in $\A$ with $A,B,C\in\X$ which induces a projective resolution
$$
\cdots \to \lhom_\A(-,\syz C)|_{\underline\X} \to \lhom_\A(-,A)|_{\underline\X} \to \lhom_\A(-,B)|_{\underline\X} \to \lhom_\A(-,C)|_{\underline\X} \to F \to 0.
$$
Since $C$ has finite projective dimension, the $n$th syzygy $\syz^nC$ is projective for some $n\ge0$.
Hence the above resolution has length less than $3n$.

$(2)\Rightarrow(1)$:
Suppose that $\X$ satisfies $(\g_n)$.
Then Theorem \ref{main3} yields triangle equivalences
$$
\zero\cong\ds(\underline\X)\cong\lgp(\mod\underline\X)=\underline{\syz^{3n}(\mod\underline\X)}
$$
for some $n\ge0$.
From this we obtain $\syz^{3n}(\mod\underline\X)=\proj(\mod\underline\X)$, which implies that $\mod\underline\X$ has global dimension at most $3n$.

$(4)\Rightarrow(1)$:
Assume that $\A$ has finitistic dimension $m$.
With the notation of the proof of the implication $(4)\Rightarrow(3)$, the syzygy $\syz^mC$ is projective, whence the length of the resolution is less than $3m$.
\end{proof}

Thanks to Proposition \ref{sgzfgl}, now we know that the stable category of a quasi-resolving subcategory is singularly equivalent to $\zero$ if and only if it is regular.
Thus, in the rest of this section we consistently use the terminology of regularity instead of singular equivalence to $\zero$.

Applying Proposition \ref{sgzfgl} to categories of modules over rings, we have the following.

\begin{cor}
\begin{enumerate}[\rm(1)]
\item
Let $R$ have finite Krull dimension.
Let $\X$ be a resolving subcategory of $\mod R$ whose objects have finite projective dimension as $R$-modules.
Then $\mod\underline\X$ has finite global dimension.
\item
Let $R$ be a local complete intersection.
Let $\X$ be a resolving subcategory of $\mod R$.
Then $\underline\X$ is regular if and only if $\mod\underline\X$ has finite global dimension.
\end{enumerate}
\end{cor}

\begin{proof}
(1) The finitistic dimension of $\mod R$ is at most the Krull dimension of $R$ (see \cite[Seconde partie, Th\'{e}or\`{e}me (3.2.6)]{RG}), so it is finite.
Combining this with Proposition \ref{sgzfgl} shows the assertion.

(2) The assertion follows from Proposition \ref{sgzfgl} and Corollary \ref{cia}(2).
\end{proof}

The following lemma plays an essential role in the proof of the main result of this section.

\begin{lem}\label{midfree}
Let $R$ be a Gorenstein Henselian local ring.
Let $\X$ be a quasi-resolving subcategory of $\mod R$ contained in $\cm(R)$ and closed under cosyzygies.
Assume that there exists a nonsplit exact sequence
$$
\sigma:0 \to X \xrightarrow{f} Y \xrightarrow{g} Z \to 0
$$
of $R$-modules with $X,Y,Z\in\X$ such that $X,Z$ are indecomposable.
If $\underline\X$ is regular, then $Y$ is free, and $X$ is isomorphic to $\syz Z$.
\end{lem}

\begin{proof}
The short exact sequence $\sigma$ induces an exact sequence
$$
\lhom_R(-,X)|_{\underline\X}\xrightarrow{\lhom_R(-,f)|_{\underline\X}}
\lhom_R(-,Y)|_{\underline\X}\xrightarrow{\lhom_R(-,g)|_{\underline\X}}
\lhom_R(-,Z)|_{\underline\X}\xrightarrow{\pi} F\to0
$$
in $\mod\underline\X$; see Proposition \ref{longex}(2).
The functor $F$ belongs to $\mod\underline\X$, and it is not isomorphic to $0$.
Indeed, if $F\cong0$ in $\mod\underline\X$, then there exists a homomorphism $h:Z\to Y$ such that $\underline{gh}=\underline{1}$.
So there are homomorphisms $\alpha:Z\to P$ and $\beta:P\to Z$ with $P$ a free $R$-module such that $1-gh=\beta\alpha$.
One has an equality $1=gh+\beta\alpha$ in the local ring $\End_R(Z)$.
Hence either $gh$ or $\beta\alpha$ is an automorphism, and in each case it is deduced that $\sigma$ splits.
This contradiction shows $F\ncong0$ in $\mod\underline\X$.

Taking advantage of Theorem \ref{main3}, we have
$$
\underline{\mod\underline\X}=\lgp(\mod\underline\X)\cong\ds(\underline\X)\cong\zero,
$$
which means that all the objects of $\mod\underline\X$ are projective.
Hence $F$ is a projective object of $\mod\underline\X$, which implies that $\pi$ is a split epimorphism.
We establish a claim.

\begin{claim*}\label{piiso}
The morphism $\pi:\lhom_R(-,Z)|_{\underline\X}\to F$ is an isomorphism.
\end{claim*}

\begin{proof}[Proof of Claim]
The claim follows from a similar argument to the proof of \cite[Lemma IV.6.5]{ASS}.
Let $\theta:F\to\lhom_R(-,Z)|_{\underline\X}$ be a splitting of $\pi$, i.e., a morphism with $\pi\theta=1$.
Then $\theta\pi:\lhom_R(-,Z)|_{\underline\X}\to\lhom_R(-,Z)|_{\underline\X}$ is represented as $\lhom_R(-,\ell)|_{\underline\X}$ for some $\ell\in\End_R(Z)$; see Remark \ref{yoneda}(1).
We have $(\theta\pi)^2=\theta(\pi\theta)\pi=\theta\pi$, which shows that $\lhom_R(-,\ell^2)|_{\underline\X}=\lhom_R(-,\ell)|_{\underline\X}$.
Hence $\underline{\ell^2}=\underline\ell$, namely, $\underline{\ell}$ is an idempotent in $\lend_R(Z)$.
Since $Z$ is a nonfree $R$-module, we have $\lend_R(Z)\ne0$, and $\lend_R(Z)$ is a local ring as $Z$ is indecomposable.
Therefore $\underline\ell$ is either $\underline0$ or $\underline1$.
If $\underline\ell=\underline0$, then $\theta\pi=0$, and $F=0$ as $\theta$ is a monomorphism and $\pi$ is an epimorphism.
So we must have $\underline\ell=\underline1$, which shows that $\pi$ is an isomorphism.
\renewcommand{\qedsymbol}{$\square$}
\end{proof}

The claim shows $\lhom_R(-,g)|_{\underline\X}=0$, and hence $\lhom_R(-,f)|_{\underline\X}:\lhom_R(-,X)|_{\underline\X}\to\lhom_R(-,Y)|_{\underline\X}$ is surjective, and so $Y$ is isomorphic in $\lcm(R)$ to a direct summand of $X$.
Since $X$ is indecomposable in $\lcm(R)$, the object $Y$ is isomorphic to either $0$ or $X$ in $\lcm(R)$.
Hence the $R$-module $Y$ is either free or stably isomorphic to $X$.

Suppose that $Y$ is stably isomorphic to $X$.
Then $\lhom_R(-,Y)|_{\underline\X}$ is isomorphic to $\lhom_R(-,X)|_{\underline\X}$, and it is seen from \cite[Theorem 2.4]{M} that $\lhom_R(-,f)|_{\underline\X}$ is an isomorphism.
Let $\lhom_R(-,\lambda)|_{\underline\X}:\lhom_R(-,Y)|_{\underline\X}\to\lhom_R(-,X)|_{\underline\X}$ be a inverse morphism of $\lhom_R(-,f)|_{\underline\X}$.
Then we have $\underline{\lambda f}=\underline1$, which implies that $\lambda f+\beta\alpha=1$ in $\End_R(X)$ for some homomorphisms $\alpha:X\to L$ and $\beta:L\to X$ with $L$ free.
Since $\End_R(X)$ is a local ring, either $\lambda f$ or $\beta\alpha$ is an automorphism, and in either case the short exact sequence $\sigma$ splits, which is a contradiction.
Consequently, the $R$-module $Y$ has to be free.
\end{proof}

For an object $X$ in an additive category $\C$ we denote by $\add_\C X$ (or simply $\add X$ if there is no confusion) the {\em additive closure} of $X$ in $\C$, that is, the full subcategory of $\C$ containing $X$ which is closed under finite direct sums and direct summands.
We say that a Cohen-Macaulay local ring $R$ has {\em finite CM-representation type} if there exist only a finite number of isomorphism classes of indecomposable MCM modules over $R$.

\begin{lem}\label{end}
Let $\C$ be an additive category, and let $X$ be an object of $\C$.
Then the assignment $F\mapsto F(X)$ makes an equivalence
$$
\mod(\add_\C X)\to\mod\End_\C(X).
$$
In particular, if $R$ is a Cohen-Macaulay local ring of finite CM-representation type, then $\lcm(R)=\add_{\lcm(R)}G$ for some MCM $R$-module $G$, and one has an exact functor $\mod\lcm(R)\to\mod\lend_R(G)$ of abelian categories which is an equivalence.
\end{lem}

\begin{proof}
Let $\Lambda$ be any ring.
Applying \cite[Proposition 2.5]{Au} to the inclusion $\{\Lambda\}\subseteq\proj\Lambda$ shows that the assignment $F\mapsto F(\Lambda)$ makes an equivalence $\mod(\proj\Lambda)\to\mod\Lambda$.
Note that the assignment $M\mapsto\Hom_\C(X,M)$ makes an equivalence $\add_\C X\cong\proj\End_\C(X)$.
Now letting $\Lambda=\End_\C(X)$ completes the proof of the first assertion.
The second assertion follows from the first assertion and the fact that any equivalence of abelian categories is an exact functor; for instance, see the proof of \cite[Proposition 21.4]{AF}.
\end{proof}

Let $R$ be a local ring.
Recall that $M$ is said to have {\em complexity} $c$, denoted by $\cx_RM=c$, if $c$ is the least nonnegative integer $n$ such that there exists a real number $r$ satisfying the inequality $\beta_i^R(M)\le ri^{n-1}$ for all $i\gg 0$.
It is known that if $R$ is a complete intersection, then the codimension of $R$ is the maximum of the complexities of $R$-modules.
For details on the complexity of a module, we refer the reader to \cite[\S4.2]{A}.

Let $R$ be a $d$-dimensional Gorenstein local ring with algebraically closed residue field $k$ of characteristic zero.
Then $R$ contains a field isomorphic to $k$, and it is known that $R$ has finite CM-representation type if and only if $R$ is a {\em simple (hypersurface) singularity} \cite[\S8]{Y}, namely, $R$ is isomorphic to a hypersurface
$$
k[[x_0,\dots,x_d]]/(f),
$$
where $f$ is one of the following.
\begin{align*}
(\a_n)\quad & x_0^2+x_1^{n+1}+x_2^2+\cdots+x_d^2,\\
(\d_n)\quad & x_0^2x_1+x_1^{n-1}+x_2^2+\cdots+x_d^2,\\
(\e_6)\quad & x_0^3+x_1^4+x_2^2+\cdots+x_d^2,\\
(\e_7)\quad & x_0^3+x_0x_1^3+x_2^2+\cdots+x_d^2,\\
(\e_8)\quad & x_0^3+x_1^5+x_2^2+\cdots+x_d^2.
\end{align*}
For each $\t\in\{\a_n,\d_n,\e_6,\e_7,\e_8\}$, a simple hypersurface singularity of type $(\t)$ is shortly called a {\em $(\t)$-singularity}.

Now we can state and prove the main result of this section, which characterizes the regularity of stable categories of resolving subcategories.

\begin{thm}\label{main2}
Let $R$ be a $d$-dimensional Gorenstein nonregular complete local ring with algebraically closed residue field $k$ of characteristic zero.
Then the following are equivalent.
\begin{enumerate}[\rm(1)]
\item
$\lcm(R)$ is regular.
\item
$R$ is a complete intersection, and $\underline\X$ is regular for every resolving subcategory $\X$ of $\mod R$.
\item
$R$ is a complete intersection, and $\underline\X$ is regular for some resolving subcategory $\X$ of $\mod R$ that contains a module of maximal complexity.
\item
$R$ is an $(\a_1)$-singularity.
\end{enumerate}
When one of these conditions holds, $\mod\lcm(R)$ has global dimension zero.
\end{thm}

\begin{proof}
(4) $\Rightarrow$ (2):
An $(\a_1)$-singularity is an isolated hypersurface singularity.
In view of Corollary \ref{either}, it is enough to show that $\lcm(R)$ is regular.
By Kn\"{o}rrer's periodicity \cite[Theorem (12.10)]{Y} we can assume that $R$ is isomorphic to either $k[[x]]/(x^2)$ or $k[[x,y]]/(xy)$.

First, let $R=k[[x]]/(x^2)$.
Then all the nonfree indecomposable $R$-modules are isomorphic to $k$.
Hence $\lcm(R)=\lmod R=\add_{\lmod R}(k)$, and we have $\mod\lcm(R)\cong\mod\lend_R(k)$ by Lemma \ref{end}.
Note that $\lend_R(k)$ is isomorphic to the field $k$, whose singularity category is $\zero$.
Therefore $\lcm(R)$ is regular.

Next, let $R=k[[x,y]]/(xy)$.
The nonisomorphic nonfree indecomposable MCM $R$-modules are $R/(x)$ and $R/(y)$, whence Lemma \ref{end} implies $\mod\lcm(R)\cong\mod\lend_R(R/(x)\oplus R/(y))$.
It is seen by using \cite[Lemma (3.9)]{Y} that
$$
\lend_R(R/(x)\oplus R/(y))
\cong\begin{pmatrix}
\lend_R(R/(x))&\lhom_R(R/(y),R/(x))\\
\lhom_R(R/(x),R/(y))&\lend_R(R/(y))
\end{pmatrix}
\cong\begin{pmatrix}
k&0\\
0&k
\end{pmatrix}
\cong k\times k,
$$
and we have $\ds(k\times k)=0$.
Thus $\lcm(R)$ is regular.

(2) $\Rightarrow$ (1):
The implication follows by letting $\X=\cm(R)$.

(1) $\Rightarrow$ (3):
Take an {\em MCM approximation}
$$
0 \to Y \to X \to k \to 0
$$
of $k$, i.e., a short exact sequence of $R$-modules such that $X$ is MCM and $Y$ has finite projective dimension; see \cite[Theorem 1.8]{ABu}.
Let $M$ be a nonfree indecomposable MCM $R$-module.
There is an exact sequence
$$
\Ext_R^1(M,X)\to\Ext_R^1(M,k)\to\Ext_R^2(M,Y),
$$
and $\Ext_R^{>0}(M,Y)=0$ since $M$ is MCM and $Y$ has finite projective dimension.
As $\Ext_R^1(M,k)$ does not vanish, neither does $\Ext_R^1(M,X)$.
Hence $\Ext_R^1(M,X')\ne0$ for some indecomposable direct summand $X'$ of $X$, and we find a nonsplit exact sequence $0\to X'\to E\to M\to0$.
As $M$ and $X'$ are MCM, so is $E$.
Applying Lemma \ref{midfree} to $\X=\cm(R)$, the module $X'$ is isomorphic to $\syz M$.
Hence $M$ is isomorphic to a direct summand of $\syz^{-1}X$, which shows that
$$
\cm(R)=\add(R\oplus\syz^{-1}X).
$$
Therefore $R$ has finite CM-representation type, and is a simple hypersurface singularity by \cite[Corollary (8.16)]{Y}.
In particular, $R$ is a complete intersection.
By assumption $\lcm(R)$ is regular, and $\cm(R)$ contains $\syz^dk$ which has maximal complexity (see also \cite[Remarks 8.1.1(2)]{A})

(3) $\Rightarrow$ (4):
Let $M$ be an $R$-module in $\X$ that has maximal complexity.
Then $\syz^dM$ is in $\X\cap\cm(R)$ and has the same complexity as $M$.
By virtue of Corollary \ref{xxcm}, replacing $\X$ with $\X\cap\cm(R)$, we may assume that $\X$ is contained in $\cm(R)$.
It follows from \cite[Corollary 4.16]{radius} that $\X$ is closed under cosyzygies.
There exists an indecomposable direct summand $N$ of $M$ having the same complexity as $M$.
Replacing $M$ with $N$, we can assume that $M$ is indecomposable.

Let $X$ be a nonfree indecomposable module in $\X$, and set $c=\codim R$.
Suppose that $\Ext_R^i(M,X)=0$ for all $1\le i\le c+1$.
Then we have $\Ext_R^{>0}(M,X)=0$ by \cite[Theorem 4.7]{AvB}, and $\Tor_{>d}^R(M,X)=0$ by \cite[Theorem III]{AvB}.
The fact that $M$ has maximal complexity forces $X$ to have finite projective dimension by \cite[Proposition 2.7]{ttp}.
This contadicts the fact that $X$ is a nonfree MCM module.
Hence $\Ext_R^\ell(M,X)\ne0$ for some $1\le\ell\le c+1$, which implies that there exists a nonsplit exact sequence
$$
0 \to X \to E \to \syz^{\ell-1}M \to 0.
$$
Note that $\syz^{\ell-1}M$ is a nonfree indecomposable module in $\X$.
Since $X$ is in $\X$, so is $E$.
Lemma \ref{midfree} shows that $X$ is isomorphic to $\syz(\syz^{\ell-1}M)=\syz^\ell M$.
It follows that the equality
$$
\X=\add(R\oplus\syz M\oplus\syz^2M\oplus\cdots\oplus\syz^{c+1}M)
$$
holds, which especially says that $\X$ is contravariantly finite in $\mod R$.
By virtue of \cite[Theorem 1.2]{arg}, we have $\X=\cm(R)$.
Hence $R$ has finite CM-representation type, and is a simple singularity.
In particular, $R$ is an isolated singularity; see \cite[Corollary 2]{HL}.

Let $C$ be a nonfree indecomposable MCM $R$-module.
Take an Auslander-Reiten sequence
$$
0 \to \tau C \to L \to C \to 0
$$
ending in $C$; see \cite[Theorem (3.2)]{Y}.
This is a nonsplit exact sequence of MCM modules with $C$ and $\tau C$ indecomposable.
Hence we can apply Lemma \ref{midfree} to see that $L$ is a free $R$-module.
Thus for each nonfree indecomposable MCM $R$-module $D$ there are no irreducible homomorphisms from $D$ to $C$ and no such ones from $\tau C$ to $D$.
This means that in the Auslander-Reiten quiver of $\cm(R)$ there is no arrow between two vertices different from $R$.
The classification of the Auslander-Reiten quivers of the MCM modules over simple singularities \cite[Chapters 8--12]{Y} together with \cite[Corollary (12.11.3)]{Y} implies that the only simple singularities $R$ where $\cm(R)$ possesses such an Auslander-Reiten quiver are $(\a_1)$-singularities.

It now follows that the conditions (1)--(4) in the theorem are equivalent.
The last assertion of the theorem is shown in the proof of the implication (4) $\Rightarrow$ (2) stated above.
\end{proof}

\begin{rem}
The condition in Theorem \ref{main2}(3) that $\X$ contains a module of maximal complexity cannot be removed.
In fact, let $R$ be any nonregular complete intersection local ring.
Let $\X$ be a resolving subcategory of $\mod R$ whose objects have finite projective dimension (e.g., the subcategory consisting of free modules, the subcategory consisting of modules of finite projective dimension, and so on).
Then the stable category $\underline\X$ is regular by Proposition \ref{sgzfgl}.
However, of course, $R$ is not necessarily an $(\a_1)$-singularity.
The reason for this is that all modules in $\X$ have complexity zero, and $R$ has positive codimension, so $\X$ does not contain a module of maximal complexity.
\end{rem}

Let $R$ be a simple hypersurface singularity.
Theorem \ref{main1} especially says that $\lcm(R)$ is not regular unless $R$ is an $(\a_1)$-singularity.
One can actually confirm this for a $1$-dimensional $(\a_2)$-singularity by direct calculation.

\begin{prop}
Let $k$ be an algebraically closed field of characteristic zero.
Let $R$ be an $(\a_2)$-singularity of dimension $1$ over $k$.
Then there is a triangle equivalence
$$
\ds(\lcm(R))\cong\ds(k[t]/(t^2)).
$$
In particular, $\lcm(R)$ is not regular.
\end{prop}

\begin{proof}
One has $R\cong k[[x_0,x_1]]/(x_0^2+x_1^3)$, and all the nonisomorphic indecomposable MCM $R$-modules are isomorphic to the maximal ideal $\m$ of $R$; see \cite[Proposition (5.11)]{Y}.
Hence $\ds(\lcm(R))$ is triangle equivalent to $\ds(\lend_R(\m))$ by Lemma \ref{end}.

The ring $R$ is isomorphic to the numerical semigroup ring $k[[t^2,t^3]]$, which is a subring of the formal power series ring $S=k[[t]]$.
Note that $R$ and $S$ have the common quotient field $K$.
One has
\begin{align*}
&\End_R(\m)\cong(\m:_K\m)=(\m:_S\m)=S,\\
&\lend_R(\m)\underset{\rm(a)}{\cong}\Tor_1^R(\tr(\m),\m)\underset{\rm(b)}{\cong}\Tor_1^R(\m,\m)\cong\Tor_2^R(\m,k)\cong k^{\oplus2},
\end{align*}
where (a) follows from \cite[Lemma (3.9)]{Y}, and (b) holds since $\tr(\m)\cong\m$.
Thus $\lend_R(\m)$ is isomorphic to a quotient ring of $S$ that has dimension $2$ as a $k$-vector space, so $\lend_R(\m)\cong S/(t^2)=k[t]/(t^2)$.
Therefore $\ds(\lcm(R))$ is triangle equivalent to $\ds(k[t]/(t^2))$, which is nonzero since $k[t]/(t^2)$ is an artinian ring of infinite global dimension.
\end{proof}

\begin{rem}
Let $R,S$ be Gorenstein local rings.
Even if $\lcm(R)$ and $\lcm(S)$ are singularly equivalent, the numbers of indecomposable MCM modules over $R$ and $S$ are not necessarily equal.
In fact, let $R=k[[x,y]]/(xy)$ where $k$ is an algebraically closed field of characteristic zero.
Then $\lcm(R)$ is singularly equivalent to $\zero=\lcm(k)$ by Theorem \ref{main2}.
The isomorphism classes of indecomposable MCM $R$-modules are those of $R$, $R/xR$ and $R/yR$, so there are three.
But $k$ is the only indecomposable MCM $k$-module.
\end{rem}

\begin{ac}
The authors are grateful to Kazuhiko Kurano for his comment on the proof of Theorem \ref{main2}, to Henning Krause for his question on Remark \ref{krause} and to Osamu Iyama for his suggestions on all parts of the paper.
The authors also thank the referee for his/her kind and helpful advice.
\end{ac}


\end{document}